\documentclass[reqno]{amsart}
\usepackage[utf8]{inputenc}
\usepackage{amsmath, amssymb, amsthm, tikz}
\usepackage{enumerate}

\usepackage{dsfont}

\usepackage{mathtools}
\mathtoolsset{showonlyrefs}

\usepackage{xcolor}
\usepackage{tikz, pgfplots}
\usetikzlibrary{arrows, shapes, positioning,decorations.markings}
\tikzstyle directed=[postaction={decorate,decoration={markings,
    mark=at position .7 with {\arrow{stealth}}}}]
\usepackage{calc}

\usepackage[colorlinks=true, pdfstartview=FitV, linkcolor=blue, citecolor=blue, urlcolor=blue,pagebackref=false]{hyperref}
\hypersetup{ colorlinks, citecolor=black, filecolor=black, linkcolor=blue, urlcolor=black }

\newtheorem{theorem}{Theorem}[section]
\newtheorem{corollary}[theorem]{Corollary}
\newtheorem{proposition}[theorem]{Proposition}
\newtheorem{lemma}[theorem]{Lemma}

\numberwithin{equation}{section}

\theoremstyle{definition}
\newtheorem{definition}[theorem]{Definition}
\theoremstyle{remark}
\newtheorem{remark}[theorem]{Remark}

\newcommand{\PP}{\mathbb{P}}

\newcommand{\EE}{\mathbb{E}}
\newcommand{\RR}{\mathbb{R}}
\newcommand{\NN}{\mathbb{N}}
\newcommand{\ZZ}{\mathbb{Z}}

\newcommand{\TT}{\mathbb T}
\newcommand{\etta}{\overline{\eta}}

\newcommand{\norm}[1]{\Vert #1 \Vert}

\newcommand{\ind}[1]{\textbf{1}_{#1}}
\newcommand{\Ind}[1]{\textbf{1}\{{#1}\}}
\newcommand{\lrp}[1]{\left(#1\right)}
\newcommand{\lrc}[1]{\left[#1\right]}
\newcommand{\lrch}[1]{\left\{ #1\right\} }
\newcommand{\lrb}[1]{\left\langle #1 \right\rangle}

\title[Non-equilibrium fluctuations via relative entropy]{Non-equilibrium fluctuations for a reaction-diffusion model via relative entropy}

\author{Milton Jara}
\address{IMPA, Estrada Dona Castorina 110, Rio de Janeiro, Brazil. }
\curraddr{}
\email{mjara@impa.br}
\thanks{}

\author{Ot\'avio Menezes}
\address{Center for Mathematical Analysis,  Geometry and Dynamical Systems,
	Instituto Superior T\'ecnico, Universidade de Lisboa,
	Av. Rovisco Pais, 1049-001 Lisboa, Portugal.}
\email{otavio.menezes@tecnico.ulisboa.pt}
\thanks{}

\begin{document}
\maketitle
\begin{abstract}
	We look at a superposition of symmetric simple exclusion and Glauber dynamics in the discrete torus in dimension 1. For this model, we prove that the fluctuations around the hydrodynamic limit are described, in the diffusive scale, by an infinite-dimensional Ornstein-Uhlenbeck process. Our proof technique is an adaptation of Yau's Relative Entropy Method that is robust enough to be adapted to other exclusion models. To cut the technical details to a minimum, we assume that the process starts from a product measure with a custom-chosen density, for which the solution of the hydrodynamic equation is stationary. Although we prove fluctuations only in dimension $ 1 $, we provide an estimate on the entropy production that holds for any dimension and a proof of the Boltzmann-Gibbs principle that applies in dimension smaller than $ 3 $. 
\end{abstract}

\tableofcontents

\section{Introduction}

This article presents a technique for studying the fluctuations around the hydrodynamic limit for some interacting particle systems out of equilibrium.
We illustrate the technique by applying it to a model where the computations are particularly simple. 
The core of the method involves bounding integrals of certain spatial averages of the system by the Dirichlet form associated to the generator. Estimates of this type come up frequently in the investigation of hydrodynamic limits and density fluctuations through variational inequalities such as the Kipnis-Varadhan inequality. They allow to trade an estimate of a functional that depends on the whole trajectory by several fixed-time estimates. Our approach combines a well-known integration by parts-like computation (see Lemma \ref{integration_by_parts}) with concentration inequalities for sums of independent random variables. To go from these estimates to the fluctuations of the density, we estimate the relative entropy between the law of the system and a product approximation.

This is an application in the fluctuations setting of \emph{Yau's relative entropy method} (\cite{yau1991relative}, \cite{kl} Chapter 6).
Yau's method is a technique for proving hydrodynamic limits. One starts with a candidate for the hydrodynamic equation and compares the evolution of the system under study with a product evolution whose parameters are given by the hydrodynamic equation. In the approximating evolution, one forgets everything about the system except the average mass at each site. It is easy to show that such product measures converge in probability to the conjectured densities, and one is left with the problem of measuring the quality of the approximation. To this end, Yau proved nequality  \eqref{yaus_ineq},
that bounds the rate of change of the entropy by an expression that depends only on the jump rates of the Markov chain. If one is able to prove inequalities of the type mentioned in the first paragraph and if the initial entropy is small, a bound on the relative entropy between the laws of the system and its pretended approximation follows. To go from such an estimate to the hydrodynamic limit, one makes use of inequality \eqref{entropy2} and large deviation estimates for the approximating measure. To go from the entropy estimate is a more difficult problem. An application of the standard inequality \eqref{entropy2} allows to bound certain additive functionals of the chain, precisely those that are not amenable to the usual variational inequalities, see Lemma \ref{bg_averaged}.
  In this article we get improved bounds on the relative entropy, see Theorem \ref{entropy}. It turns out that the method is robust, because the error terms in the upper bound for the entropy production can be computed explicitly in terms of the adjoint generator, see \eqref{yaus_ineq} and \eqref{eq_adjointformula}. Besides, once one has chosen an appropriate candidate for the approximate measure, the proof of the entropy bound runs without any further input from the model. 

\section{Notation and Results}

\subsection{The reaction-diffusion process and the starting measure}

In the present article we analyse a particle system on the $ d- $dimensional torus $ \TT^d_n = \ZZ^d/n\ZZ^d $ whose dynamics is a superposition of simple symmetric exclusion and a birth-and-death dynamics. Given a configuration $ \eta \in \{0,1\}^{\TT^d_n} $, define the rates 

\begin{equation}\label{birthdeathrate}
c_x(\eta)=\eta_x\cdot c_x^-(\eta) + (1-\eta_x)\cdot c_x^+(\eta),
\end{equation}
where
\begin{equation}\label{bd_rates}
\begin{aligned}
& c_x^-(\eta)=1,\\ 
& c_x^+(\eta)= 1+\lambda \sum_{j=1}^d\eta_{x-e_j}\eta_{x+e_j}
\end{aligned}
\end{equation}
and $ \lambda > 0 $ is a positive parameter. Some estimates, such as Theorem \ref{entropy} and Proposition \ref{bg}, are valid for arbitrary finite-range rates. The \emph{reaction-diffusion process} is the Markov process $ (\eta^n(s))_{s\geq 0} $ taking values in $ \{0,1\}^{\TT^d_n} $ with infinitesimal generator
\begin{equation}\label{rdgenerator}
L_nf:= n^2L^{ex} + L^r,
\end{equation}
where
\begin{equation}\label{key}
L^{ex}f(\eta)=\sum_{x\in\TT_n^d}\sum_{j=1}^d[f(\eta^{x,x+e_j})-f(\eta)]
\end{equation}
and
\begin{equation}\label{key}
L^rf(\eta)=\sum_{x\in \TT^d_n}c_x(\eta)[f(\eta^x)-f(\eta)],
\end{equation}
with the rates $ c_x(\eta) $ as in \eqref{birthdeathrate}.

The model was introduced in \cite{dmfl86}. In this article, the authors proved that the hydrodynamic equation of the system is a heat equation with a forcing term, $F(\rho):=\int c_x(\eta)\,\nu_{\rho}(d\eta) $.

\begin{equation}\label{hydroeqrd}
\left\{
\begin{array}{rlll}
\partial_t \rho(t,u) & = & \partial_{uu}\rho(t,u) + F(\rho(t,u)) & \mbox{ for all }t\in [0,T], u\in\TT; \\
\rho(0,u) & = & \rho_0(u) & \mbox{ for all }u\in\TT.
\end{array}
\right.
\end{equation}

In the same article, the authors prove convergence of the density fluctuation field under the stationary measure.

\subsection{Entropy estimate and density fluctuations}
\begin{theorem}[Entropy Estimate]\label{entropy}
  For each $ n\in \NN $, let $ \{\eta^n_t:t\in [0,T]\} $ denote the
  reaction-difusion process in $ \TT_n^d $ with generator
  \eqref{rdgenerator}. Let $ \rho\in (0,1) $ satisfy $ F(\rho)=0 $,
  where $ F $ is the forcing term in the hydrodynamic equation:
  \begin{equation}\label{key}
    F(m):=\int \lrch{(1-\eta_0)c_0^+(\eta) - \eta_0 c_0^-(\eta)}\,\mathrm{d}\nu_m.
  \end{equation}
  Then, there exists a constant $ C>0 $ such that
  \begin{equation}\label{key}
    \partial_t H(\eta^n_t|\nu_\rho) \leq C n^{d-2}\cdot g_d(n),
  \end{equation}
  where
  \begin{equation}\label{defgd}
    g_d(n):=
    \begin{cases}
      n&, d=1;\\
      \log n &, d=2;\\
      1 &, d\geq 3.
    \end{cases}
  \end{equation}
 
  In particular, when the system starts from the product measure
  $ \nu_{\rho} $ the following bound holds:
  \begin{equation}\label{key}
    H(\eta^n_t|\nu_{\rho}) \leq Ct \cdot g_d(n).
  \end{equation}
\end{theorem}

The random measures and its limits that appear in the statement of the fluctuation theorem belong to certain $ L^2-$based Sobolev spaces. We refer the reader to \cite{kl}, page 288, for the definitions of Sobolev spaces and white noise.

\begin{theorem}[Fluctuations]\label{fluctuations}
  Fix $ T>0 $. For each $ n\in \NN $, let $ \{\eta^n_t:t\in [0,T]\} $
  denote the reaction-difusion process in $ \TT_n $ (dimension $ 1 $)
  with generator \eqref{rdgenerator}. Let $ \rho\in (0,1) $ satisfy
  $ F(\rho)=0 $, where $ F $ is the forcing term in the hydrodynamic
  equation \eqref{hydroeqrd}.  Assume $ \eta^n_0 $ has law
  $ \nu_{\rho} $ and define the density fluctuation field by
  \begin{equation}\label{key}
    X^n_t(f):=n^{-1/2}\sum_{x\in \TT_n} f\lrp{\frac{x}{n}}\lrp{\eta^n_x(t)-\rho},
  \end{equation}
  for $ t\in [0,T] $ and $ f:\TT\to \RR $ smooth.
	
  Then the sequence $ \{X^n_t:t\in [0,T] \}_{n\in \NN} $ converges to
  the unique solution of the infinite-dimensional Ornstein-Uhlenbeck
  equation
	
	\begin{equation}
          dX_t = \lrp{\Delta -\lrp{ \frac{1}{1-\rho} - \frac{\lambda \rho^2}{1 + \lambda \rho^2} }   }X_t\,\mathrm{d}t  + \nabla \dot{W}_t,
	\end{equation}
	where $ \dot{W} $ denotes space-time white noise and the convergence under consideration is with respect to
        the $ J_1 $-Skorohod topology on the Sobolev space
        $ \mathcal{H}_{-2}(\TT) $.

        \bigskip
	
        In more detail: given a smooth function $ f: \TT \to \RR $ and
        $ t\in [0,T] $, it holds
	\begin{enumerate}
        \item The sequence of process
          $ \{X^n_t(f):t\in [0,T]\}_{n\in\NN} $ is tight in the
          $ J_1 $-Skorohod topology of $ D([0,T];\RR) $.
        \item If $ X(f) $ is a limit point, then the processes
		
		\begin{equation}\label{mart_problem_1}
                  M_t(f):=X_t(f)-X_0(f)-\int_0^t X_s(\Delta f - [1+\lambda(1-\rho)]f)\,\mathrm{d}s
		\end{equation}
		and
		\begin{equation}\label{marti_problem_2}
                  N_t(f):=M_t(f)^2 - 2t\rho(1-\rho)\,\norm{\nabla f}^2_{L^2(\TT)}
		\end{equation}
		are martingales with respect to the filtration
                $\mathcal{F}_t:=\sigma\{X_s(g):s\leq t \mbox{ and }
                g\in C^{\infty}(\TT)\}$.
              \end{enumerate}
            \end{theorem}

\begin{remark} 
  We compute the coefficients of the limiting equation in Proposition \ref{bg_geral}. The Laplacian term comes from the
  exclusion dynamics and the forcing term comes from the
  birth-and-death dynamics.
\end{remark}

\begin{remark}
  The quality of the approximation is measured by relative entropy.
  We know the approximation is good because the degree $ 1 $ term in
  the formula for the adjoint generator vanishes.  It is to the
  adjoint generator that we look to find a good candidate for the
  approximating measure.
\end{remark}

%
%
%
%
%
%
%
%

\subsection{Structure of the fluctuations proof}
There is a general framework for proving convergence results such as Theorem \ref{fluctuations}, but each model presents its own challenges. Now we lay out this general framework.

\bigskip
\noindent\textbf{Step 1: Martingale decomposition and convergence of the martingale part}

Let $f:\TT\to\RR$ be a smooth function. Define the process $\{M^n_t(f),t\in[0,T]\}$ by 

\begin{equation}\label{mart_dec_rd} 
X^n_t(f)=X^n_0(f) + M^n_t(f) + \int_0^t L_n X^n_s(f)\,\mathrm{d}s.
\end{equation}
then $M^n(f)$ is a martingale with respect to the natural filtration. The predictable quadratic variation of $M^n(f)$ is given by
	
	\begin{equation}
	\begin{split}
	\lrb{M^n_t(f)} 
	& = \int_0^t n^2\sum_{x\in\TT_n}\frac{1}{n}\left\{f\left(\frac{x+1}{n}\right) - f\left(\frac{x}{n}\right)\right\}^2(\eta_x(s) -\eta_{x+1}(s))^2\,\mathrm{d}s
	\\
	& + \int_0^t c_x(\eta(s))\sum_{x\in\TT_n}\frac{1}{n}f\left(\frac{x}{n}\right)^2\,\mathrm{d}s,
	\end{split}
	\end{equation}
	where $c_x(\eta)=\eta_x + (1-\eta_x)(1+\lambda \eta_{x-1}\eta_{x+1})$. Moreover, 
	
	\begin{equation}\label{convergence_quad_var}
	\lim_{n\to\infty}\lrb{M^n_t(f)} = 2t\,\rho(1-\rho)\norm{\nabla f}^2_{L^2(\TT)}.
	\end{equation}
When the particle system starts from equilibrium, one can prove \eqref{convergence_quad_var} using the Cauchy-Shwarz inequality. Out of equilibrium, one needs to use the entropy bound from Theorem \ref{entropy}. The convergence in \eqref{convergence_quad_var} follows from Lemma \ref{lemma_quadvar}.

Once we have \eqref{convergence_quad_var}, a direct application of the Martingale Functional Central Limit Theorem (a good reference is \cite{whitt2007proofs}, Theorem 2.1) gives convergence of the sequence $\{M^n_t:t\in[0,T]\}$  with respect to the $J_1$-Skorohod topology of $D_{[0,T]}\RR$ to a Brownian motion of covariance $2\rho (1-\rho)\norm{\nabla f}^2_{L^2(\TT)}$.

\bigskip
\noindent\textbf{Step 2: Closing the martingale}

Fix a smooth function $f:\TT\to\RR$. Assume that we have tightness for the sequence $\{X^n_t:t\in[0,T]\}_{n\in\NN}$. If the term $L_n X^n_s(f)$ inside the integral in \eqref{mart_dec_rd} were a function of $X^n$, say $ L_nX^n_s(f)= X^n_s(Bf)$ for some operator $B$, then we could pass to the limit and arrive at a martingale problem. The next proposition replaces $L_n X^n_t(f)$ by a function of $X^n$, asymptotically.

\begin{proposition}\label{bg_geral}
	Let $f:\TT \to \RR$ be a smooth funtion and $\delta > 0$.Then, for all $t\in [0,T]$,
	\begin{equation}\label{bg_eq}
	\lim_{n\to \infty}\PP_{\nu_{\rho}}\left(
	\Big|\int_0^t L_n X^n_s(f) - X^n_s \lrp{\lrc{\Delta + \frac{1}{1-\rho} - \frac{\lambda \rho^2}{1 + \lambda \rho^2}  }f}\,\mathrm{d}s \Big| > \delta
	\right )=0.
	\end{equation}
\end{proposition}

\begin{proof}
The first step is to write $  L_n X^n_s(f)$ in terms of the variables $ \etta_x:=\eta_x - \rho $. Using the shorthand $ f_x := f\lrp{\frac{x}{n}} $, we can compute
\begin{equation}\label{formula_dynkin}
\begin{aligned}
L_nX^n(f) &=
X^n\lrp{\Delta_n f} \\ &+ \frac{1}{\sqrt n}\sum_{x\in \TT_n} \lrc{ 
-(2+\lambda \rho^2)f_x + \lambda \rho(1-\rho) (f_{x+1} + f_{x-1})
}\etta_x \\
&- \frac{1}{\sqrt n}\sum_{x\in \TT_n}  \lambda\rho( f_{x-1}  + f_{x+1} )
\etta_{x-1}\etta_x 
+\frac{1}{\sqrt n}\sum_{x\in \TT_n} \lambda(1-\rho) f_{x-1}\etta_{x-2}\etta_x\\
&-\frac{1}{\sqrt n}\sum_{x\in \TT_n}\lambda f_{x-1} \etta_{x-2}\etta_{x-1}\etta_x,
\end{aligned}
\end{equation}
where, as usual, $ \Delta_n f(u):= n^2(f(u+n^{-1}) + f(u-n^{-1}) - 2f(u) ) $ is an approximation to the second derivative. It is possible to show that the first term converges, after integration in time, to $ X^n(\Delta f) $.

The hardest part of the proof of Theorem \ref{fluctuations} is to show that the last three terms vanish in the limit, after integration in time. This statement is know as the Boltzmann-Gibbs Principle. Its proof takes the whole of Section \ref{section_bg} and uses the entropy estimate \ref{entropy} as an essential ingredient. 

In the linear term, it is possible to replace $ f_{x\pm 1} $ by $ f_x $ with an error of order $ n^{-1/2} $. The second term in the above sum is equal to $ X^n\lrp{ \lrp{ \lambda\rho(1-\rho) - 2 - \lambda \rho^2  } f} $ plus a negligible term. Notice that this works only in dimension $ 1 $. To arrive at the coefficient $ -\lrp{ \frac{1}{1-\rho} - \frac{\lambda \rho^2}{1 + \lambda \rho^2} } $, we make use of the identity that defines $ \rho $, that is $ (1-\rho)(1+\lambda \rho^2)=\rho $.

\end{proof}

\bigskip
\noindent\textbf{Step 3: Tightness of the additive functional process}

In Section \ref{section_tightness}, we prove that, for every smooth test function $f:\TT\to\RR$, the sequence of additive functionals 

\begin{equation*}
\left\{\int_0^t L_nX^n_s(f)\,\mathrm{d}s:t\in[0,T]\right\}_{n\in\NN} 
\end{equation*}
is tight in $C([0,T];\RR)$. We have already seen that the sequence of martingales $\{M^n(f)\}_{n\in\NN}$ converges. An application of Mitoma's Theorem (\cite{mitoma1983tightness}, Theorem 3.1) yields then tightness of the distribution-valued sequence $\{X^n_t:t\in[0,T]\}_{n\in\NN}$. 

\bigskip
\noindent\textbf{Step 4: Putting the proof together}

In \cite{holley1978generalized} it is proven that the martingale problem defined by \eqref{mart_problem_1} and \eqref{marti_problem_2} has only one solution. We have to verify that the limit points of the sequence $\{X^n_t:t\in [0,T]\}_{n\in\NN}$ are solutions to this martingale problem and find the law of $X_0$. 

By Proposition \ref{bg_geral}, $M^n_t(f)$ has the same limit as the sequence

\begin{equation}
\tilde{M}^n_t(f):=X^n_t(f)-X^n_0(f) - \int_0^t X^n_s\lrp{\Delta f -\lrp{ \frac{1}{1-\rho} - \frac{\lambda \rho^2}{1 + \lambda \rho^2} } f}\,\mathrm{d}s.
\end{equation}

As we remarked in Step 2, it follows from the Martingale FCLT that $M^n(f)$ converges to a Brownian motion of variance $2t\rho(1-\rho)\norm{\nabla f}_{L^2(\TT)}^2$. This verifies that the limit points solve the martingale problem given by \eqref{mart_problem_1} and $\eqref{marti_problem_2}$.

It remains to determine the law of $X_0$. Since the initial distribution is product, the computation with characteristic functions in \cite{kl}, Corollary 11.2.2, works. We discover that the random field $X_0$ is centered Gaussian and its covariances are given by $\EE[X_0(f)X_0(g)]=\rho(1-\rho)\int_{\TT}f(u)g(u)\,\mathrm{d}u$.

\section{Model independent results}

We collect in the present section statements that depend on the model only through our choice of reference measure.

\subsection{Definitions}

Throughout this section we work with a continuous-time Markov chain $ \{X_t:t\geq 0\} $ on the finite state space $ \Omega $, with infinitesimal generator $ L $ that acts on funcions $ f:\Omega \to \RR $ as 

\begin{equation}\label{key}
Lf(x):=\sum_{y\neq x}r(x,y)\lrc{f(y)-f(x)},
\end{equation}
and the transition rates $ \{r(x,y):(x,y)\in \Omega\times \Omega\} $ are non-negative. 

The \emph{carr\'e du champ} operator associated to $ L $ is the bilinear operator $ \Gamma: \RR^{\Omega}\times \RR^{\Omega} \to \RR^{\Omega} $ defined by 

\begin{equation}\label{key}
\Gamma(f,g)(x):=\sum_{y\in\Omega}r(x,y)\lrp{ f(y)-f(x) }\lrp{g(y)-g(x)}.
\end{equation}

We denote $ \Gamma(f,f) $ simply by $ \Gamma(f) $.

Given two probability measures $ \mu $ and $ \nu $ in $ \Omega $ with $ \mu $ absolutely continuous with respect to $ \mu $, we denote by $ H(\mu|\nu) $ the \emph{relative entropy} between $ \mu $ and $ \nu $:

\begin{equation}\label{key}
H(\mu|\nu) = \sum_{x\in \Omega} \frac{\mu(x)}{\nu(x)}\log \frac{\mu(x)}{\nu(x)} \nu(x).
\end{equation}

\subsection{Inequalities}

\begin{proposition}[Yau's Inequality] For $ t>0 $, let $ \mu_t $ denote the law of $ X_t $. Let  $ \nu $ and $ \nu_t $ be arbitrary probability measures in $ \Omega $, with the sole restrictions that $ \mu_t $ be absolutely continuous with respect to $ \nu_t $ and $ \nu_t $ be absolutely continuous with respect to $ \nu $. 
	Consider the densities
	
	\begin{equation}\label{key}
	g_t(x):=\frac{\mu_t(x)}{\nu_t(x)}
	\mbox{ and }
	\psi_t(x):=\frac{\nu_t(x)}{\nu(x)}.
	\end{equation}
	Let $ L^*_t $ denote the adjoint of the generator $ L $ in $ L^2(\nu_t) $. Then the following inequality holds:
	\begin{equation}\label{yaus_ineq}
	\partial_t H(\mu_t|\nu_t) \leq \int 
	\lrp{ L^*_t\ind{} - \partial_t \log \psi_t} g_t - \Gamma\lrp{\sqrt g_t}\,\mathrm{d}\nu_t. 
	\end{equation}
\end{proposition}

\begin{remark}
	Notice that the quantity $ \int \partial_t \log \psi_t\,d\nu_t $ does not depend on $ \nu $. In fact, this term equals $ \int g_t\partial_t\psi_t\,d\nu = - \int \partial_t g_t \,d\nu_t  $. 
\end{remark}

\begin{proof}
First of all we make the observation that, because our state space is finite, all integrals involved are actually finite sums, so interchanges of integrals and derivatives are automatically justified. 
For all $ f:\Omega\to\RR $, we have \footnote{This is another way of writing $\partial_t\EE_{\nu}[f(X_t)]=\EE_{\nu}[Lf(X_t)]$.}
\begin{equation}\label{key}
\int \partial_t(\psi_t g_t)\cdot f\,d\nu = \int \psi_t g_t\cdot Lf\,d\nu.
\end{equation}

Therefore

\begin{equation*}
\begin{split}
\partial_t H(\mu_t|\nu_t) 
&= \partial_t  \int \psi_t g_t \cdot\log g_t \,d\nu\\
&= \int g_t \cdot\,L \log g_t \,d\nu_t
+ \int \psi_t \cdot \,\partial_t g_t\,d\nu\\
& =\int g_t \cdot \,L \log g_t \,d\nu_t + \int \partial_t(\psi_t g_t) - g_t\cdot\partial_t \psi_t \,d\nu
\end{split}
\end{equation*}

The second integral is equal to $-\int g_t\frac{\partial_t \psi_t}{\psi_t}\,d\nu_t = - \int \partial_t\log \psi_t \,d\mu_t$. It remains to show

\begin{equation*}
\int g_t \cdot L \log g_t \,d\nu_t \leq - \Gamma\lrp{\sqrt {g_t}} + \int Lg_t \,d\nu_t.
\end{equation*}

For that, we write down $L\log g_t$ with all its jump rates and use the elementary inequality $a(\log b - \log a)\leq 2 \sqrt{a}(\sqrt{b}-\sqrt{a})$, that is true for any positive $ a $ and $ b $. The result is 

\begin{equation}\label{yau_inequality2}
\int g_t \cdot L\log g_t \,d\nu_t \leq 2 \int \sum_{y\in\Omega} r(x,y) \sqrt{g_t(x)}\lrp{ \sqrt{g_t(y)}-\sqrt{g_t(x)} }\,d\nu_t(x).
\end{equation}

To finish, we use the identity $2\sqrt a(\sqrt b - \sqrt a) = -(\sqrt b - \sqrt a)^2 + (b-a)$.
\end{proof}

\begin{proposition}\label{adjointformula}
	Let $ \nu $ be a probability measure on the state space $ \Omega$. Denote by $ L^* $ the adjoint of $ L $ in $ L^2(\nu) $. Let $ L^* $ denote the ajoint of of the generator $ L $ in $ L^2(\nu) $. Then
\end{proposition}                    

\begin{equation}\label{eq_adjointformula}
L^*\ind{}(x) = \sum_{y\neq x}\lrch{\frac{\nu(y)}{\nu(x)}r(y,x)-r(x,y)}.
\end{equation}

\begin{proof} We begin by computing
\begin{equation}\label{key}
\begin{aligned}
L^*f(x) 
&= \frac{1}{\nu(x)}\int \ind{x} \cdot L^* f \,\mathrm{d}\nu \\
&= \frac{1}{\nu(x)}\int L\ind{x} \cdot  f \,\mathrm{d}\nu.
\end{aligned}
\end{equation}

Thus we need to compute $ L\ind{x} $. We have

\begin{equation}\label{key}
L\ind{x}(y) = \left\{
\begin{array}{lr}
r(y,x) & ,y\neq x\\
-\sum_{z\neq x}r(x,z) & ,y=x.
\end{array}
\right.
\end{equation}

Substituting into the previous formula, we get 

\begin{equation}\label{lestrelaf}
L^*f(x)= \frac{1}{\nu(x)}\sum_{y\neq x} \lrch{\nu(y)r(y,x)f(y) - \nu(x)r(x,y)f(x)}.
\end{equation}

Taking $ f = 1 $, we finish the proof.
\end{proof}

\begin{lemma}[Integration by parts]\label{integration_by_parts} Given $ x,y\in \TT^d_n $ and $ \eta \in \{0,1\}^{\TT^d_n} $, denote by $ \eta^{x,y} $ the 
	configuration that exchanges the values of $ \eta_x $ and $ \eta_y $.
	
	Let $g$ and $h$ be functions on the configuration space $\{0,1\}^{\TT_n^d}$. Assume $h$ is invariant under the change of variables $\eta\mapsto \eta^{x,y}$. Then, for any positive $a$, the following inequality holds:
	\begin{equation}
	\int g\cdot h(\eta_x - \eta_y)\,\mathrm{d}\nu_{\rho} \leq a n^2 \int \lrp{
		\sqrt{g\lrp{\eta^{x,y}}} -\sqrt {g(\eta)}
	}^2 
	\,\mathrm{d}\nu_\rho(\eta) + \frac{1}{an^2}\int h^2\cdot g\,\mathrm{d}\nu_{\rho}.
	\end{equation}
\end{lemma}

\begin{proof}
	Denote $ g^{x,y}(\eta):=g(\eta^{x,y}) $. Since $\nu_{\rho}$ is invariant under the change of variables $\eta\mapsto \eta^{x,y}$,
	
	\begin{equation}
	\begin{split}
	\int g\cdot h (\eta_x - \eta_y)\,\mathrm{d}\nu_{\rho} &= \frac{1}{2}\int 
	h(g - g^{x,y})(\eta_x - \eta_y)\,\mathrm{d}\nu_{\rho}.
	\end{split}
	\end{equation}
	
	Now we factor $g - g^{x,y} = (\sqrt g -\sqrt g^{x,y})(\sqrt g + \sqrt g^{x,y})$ and apply the elementary inequality $uv\leq 2an^2 u^2 + \frac{v^2}{2an^2}$. To finish the proof, we use $(\sqrt g^{x,y} +\sqrt g)^ 2\leq 2(g^{x,y}+g)$ and recall that $h^{x,y}=h$ by assumption.

\end{proof}

\begin{lemma}[Feynman-Kac Inequality]\label{fkineq}
	Let $ (x_t)_{t\geq 0} $ be a Markov chain on the finite state space $ \Omega $, with infinitesimal generator $ L $. 
	
	Let $ \nu $  be a probability measure in $ \Omega $. Consider  an arbitrary family
	$ (\mu_t)_{t\geq 0} $ of probability measures in $ \Omega $, absolutely continuous with respect to $ \nu $, and denote the densities by $ \psi_t:=\frac{d\mu_t}{d\nu} $.
	
	Assume the law of $ x_0 $ is $ \mu_0 $. 
	Let $ W:\RR_+\times \Omega \to \RR $ be a bounded function and fix $ t>0 $. Then
	
	\begin{equation}\label{key}
	\log\EE_{\mu_0} \lrc{
		\exp\lrch{
			\int_0^t W(x_s)\,\mathrm{d}s}} \leq \int_0^t C_s\,\mathrm{d}s,
	\end{equation} 
	where
	\begin{equation}\label{key}
	C_s:= \sup_{g}\lrch{
		\int \lrp{
			W_s + \frac{1}{2}L-\frac{1}{2}\partial_s \log \psi_s
		}g - \frac{1}{2}\Gamma\lrp{\sqrt g}\,\mathrm{d}\mu_s
	},
	\end{equation}
	the supremum being taken over the set of all $ \mu_s $-densities $ g:\Omega\to \RR $, that is, $ g\geq 0 $ and $ \int g\,\mathrm{d}\mu_s=1 $.
	
\end{lemma}

\begin{remark}
	This is an extension of Lemma 7.2 in Appendix A.1 of \cite{kl}. The version we stated above is useful when the approximating measures change with time.
\end{remark}

\begin{remark}
	The most delicate estimates we need to do involve temporal cancellations. The only robust tools we are aware of to deal with temporal cancellations are Feynman-Kac's Inequality and Kipnis-Varadhan's Inequality. Both methods are primarily analytical, and the mechanism by which they account for the temporal cancellations is not clear to us. 
\end{remark}

\begin{proof}
	We claim that there exists a function $ h:[0,t]\times \Omega \to \RR $ such that $ h(t,x) = 1 $, 
	\begin{equation}\label{fk_formula}
	h(0,x) = \EE_x \lrc{
		\exp\lrch{
			\int_0^t W(s,x_s)\,\mathrm{d}s
	}}.
	\end{equation}
	and that satisfies the equation 
	\begin{equation}\label{key}
	\partial_s h(s,x) = - (Lh)(s,x) - W(s,x)h(s,x).
	\end{equation}
	The reader familiar with the Feynman-Kac formula can already guess what the function should be. For the moment, let's just assume it exists and use it to prove our inequality. 
	Denote by $ f_s:\Omega \to \RR_+ $ the Radon-Nykodym density of the law of $ x_s $ 
	with respect to $ \mu_s $. We want to bound $ \int h_0 \,\mathrm{d}\mu_0 $. Define then
	\begin{equation}\label{key}
	\phi(s):=\int h_s^2  \,\mathrm{d}\mu_s = \int h_s^2  \psi_s \,\mathrm{d}\nu.
	\end{equation}
	By Gronwall's inequality, it is enough to prove
	\begin{equation}\label{fk_gronwall}
	-\phi '(s) \leq 2C_s \phi(s).
	\end{equation}
	Differentiating, we get 
	\begin{equation}\label{key}
	\phi '(s) = \int 2 h_s\partial_s h_s + h_s^2 \partial_s (\log \psi_s)\,\mathrm{d}\mu_s.
	\end{equation}
	Recall that, by assumption, $ \partial_s h_s = -(L + W_s)h_s $. Then, denoting $ \overline{h_s}:=h_s\lrp{\int h_s^2\,\mathrm{d}\mu_s}^{-1} $ we get 
	\begin{equation}\label{key}
	\begin{aligned}
	-\phi'(s) &= \phi(s)\cdot\int 2\overline{h_s} (L+W_s) \overline{h_s} - \overline{h_s}^2 \partial_s\log \psi_s \,\mathrm{d}\mu_s  \\
	& \leq 2\phi(s)\cdot 
	\sup_{h\geq 0, \int h^2\,\mathrm{d}\mu_s=1} \lrch{		
		\int h (L + W)h - \frac{1}{2}h^2 \partial_s \log\psi_s\,\mathrm{d}\mu_s
	}.
	\end{aligned}
	\end{equation}
	Using $ L(h^2) = 2 hLh + \Gamma(h) $, we see that the supremum above is equal to $ C_s $ and thus finish the proof of \eqref{fk_gronwall}.
	
	It remains to show that the function $ h:[0,t]\times \Omega \to \RR $ that solves the backward equation
	\begin{equation}
	\left\{
	\begin{array}{rlr}
	(\partial_s + L)h(s,x) &= -W(s,x)h(s,x) & x\in \Omega, s\in[0,t],\\
	h(t,x) &= 1 & x\in\Omega
	\end{array}\right.
	\end{equation}
	satisfies \eqref{fk_formula}.

	Recall the exponential Dynkin martingales: for any bounded $ g:\RR_+\times \Omega \to \RR $, the process
	
	\begin{equation}\label{key}
	M^g_t:= \exp\lrch{
		g(t,x_t) - g(0,x_0) - \int_0^t e^{-g(s,x_s)}(\partial_s + L)e^{g(s,x_s)}\,\mathrm{d}s
	}
	\end{equation}
	is a martingale with respect to the natural filtration. Taking $ g = \log h $, we find that the process
	\begin{equation}\label{key}
	\lrch{M^g_s:=
		\frac{h(s,x_s)}{h(0,x_0)}\exp \lrch{\int_0^s W(r,x_r)\,\mathrm{d}r}
		: s\in [0,t] }
	\end{equation}
	is a mean one martingale, so that 
	\begin{equation}\label{key}
	h(0,x) = \EE_x[M^g_t] = \EE_x \lrc{
		\exp \lrch{\int_0^t
			W(r,x_r)\,\mathrm{d}r
		} 
	},
	\end{equation}
	as we wanted to show.

\end{proof}

\subsection{Mass transport and flows}
We think of telescoping sums as mass transport. The trivial identity 

\begin{equation*}
\eta_0 - \eta_{\ell} = \sum_{j=1}^{\ell} \eta_{j-1}-\eta_{j}
\end{equation*}
describes the movement of a point mass from $0$  to  $\ell$ in $\ell$ steps: at step $j$, mass $1$ goes from $j-1$ to $j$. A less obvious identity (used in the proof of the Replacement Lemma) is 

\begin{equation}
\eta_0 - \frac{\eta_1+\cdots + \eta_\ell}{\ell} = \sum_{j=0}^{\ell-1}\frac{\ell - j }{\ell}(\eta_{j} - \eta_{j+1}). 
\end{equation} 
Here one spreads a unit mass at $0$ uniformly along the interval $\{1,\ldots, \ell\}$ by sending mass $1$ from $0$ to $1$, mass $\frac{\ell - 1}{\ell}$ from $1$ to $2$, mass $ \frac{\ell - 2}{\ell} $ from $ 2 $ to $ 3 $ and so on.  
In $d$ dimensions, we have a similar identity. Let $\ell\in \NN$ and $\Lambda_\ell:=\{y\in \ZZ^d: 0\leq y < \ell\}$.
In Lemma \ref{flux_lemma} below, we find a function $\phi: \Lambda_{\ell}\to \RR$ that satisfies

\begin{equation}
\eta_0 - \frac{1}{\ell^d}\sum_{y\in \Lambda_{\ell}}\eta_y = \sum_{j=1}^d\sum_{0\leq y < \ell}\phi_{y}(\eta_y - \eta_{y+e_j})
\end{equation}
and such that $\sum_y\phi_y^2$ is small. 

\begin{definition}\label{flowdef}
	Given two measures $\mu$ and $\nu$ on the finite set $\Omega$, we say that $\phi:\Omega\times \Omega\to \RR$ is a \emph{flow connecting $\mu$ and $\nu$}, and write $ \phi: \mu \mapsto \nu $, if
	\begin{enumerate}[(i)]
		\item $\phi(x,y)=-\phi(y,x)$ for all $x,y \in \Omega$;
		\item $\sum_{y\in \Omega} \phi(x,y)=\nu(x)-\mu(x)$.
	\end{enumerate}
	We call \emph{support} of $\phi$ the set of oriented edges $\{(x,y)\in \Omega\times \Omega: \phi(x,y)\neq 0 \}$, and refer to as \emph{cost} or \emph{norm} of $\phi$ the quantity $\norm{\phi}^2:=\sum_{x,y\in \Omega}\phi(x,y)^2$.
\end{definition}

Our goal is to construct a flow in a box of $\ZZ^d$ that connects the point mass to the uniform distribution at small cost.

\begin{theorem}[Flow Lemma]\label{flux_lemma}
	Let $d$ and $\ell$ be positive integers. Let $\Lambda_\ell:=\{1,\ldots, \ell\}^d$. 
	
	Then, there exists a flow $\phi^\ell:\Lambda_{\ell} \to \RR$ that connects the point mass at $(1,\ldots,1)$ to the uniform distribution in $\Lambda_\ell$ and is supported in nearest neighbour edges such that $ \norm{\phi^\ell}^2 = O\lrp{g_d(\ell)} $, with $ g_d $ as defined in \eqref{defgd}.
	\footnote{By $\norm{\phi}^2=O(\ell)$, we mean that $\norm{\phi^\ell}^2 \leq Cg_d(\ell)$ for some constant $C$ that does not depend on $\ell$. Similarly for the other two bounds.}
	In addition, there is a flow that connects the point mass at zero to the uniform distribution in $ \Lambda $ whose cost is of the same order.
\end{theorem}

\begin{remark}
	The concept of mass flow on a graph is closely related to that of current flow in electric networks. Indeed, consider an electric network where every edge has resistance $1$. If $a$ and $z$ are distinct nodes of that network then a unit current flowing from $a$ to $z$ is also a mass flow connecting the point mass at $a$ to the point mass at $z$. 
\end{remark}

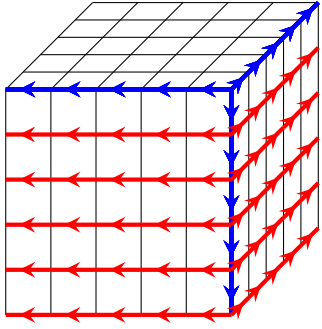
\begin{figure}
	\centering
\begin{tikzpicture}[scale=.6] 
	\clip (-3,-3,-3) rectangle (10.5, 10.5, 10.5);
	\foreach \x in{1,...,6}
	{   \draw (1,\x ,6) -- (6,\x ,6);
		\draw (\x ,1,6) -- (\x ,6,6);
		\draw (6,\x ,6) -- (6,\x ,1);
		\draw (\x ,6,6) -- (\x ,6,1);
		\draw (6,1,\x ) -- (6,6,\x );
		\draw (1,6,\x ) -- (6,6,\x );
	}
	\foreach \x in{1,...,5}
	\foreach \y in{1,...,5}
	{
		\draw [blue, ultra thick, directed] (\x+1,6,6) -- (\x,6,6);
		\draw [blue, ultra thick, directed] (6,\x+1,6) -- (6,\x,6);
		\draw [blue, ultra thick, directed] (6,6,\x+1) -- (6,6,\x);
		\draw [red, ultra thick, directed] (6,\y,\x+1) -- (6,\y,\x);
		\draw [red, ultra thick, directed] (\x+1,\y,6) -- (\x,\y,6);
}
\end{tikzpicture}
\caption{The blue flow spreads the mass from the corner of the cube along its edges. The red flow spreads the mass from one of the edges of across the faces of the cube adjacent to it.}
\end{figure}

In the remaining of the present subsection, we are going to prove Theorem \ref{flux_lemma}. Our proof is going to be constructive. In one dimension, one can take
\begin{equation}\label{1dflux}
\phi(k,k+1):=\frac{\ell - k}{\ell}\Ind{0\leq k < \ell}. 
\end{equation}

In higher dimensions, we will not give an explicit formula for the flow, but will build it instead by gluing together several copies of \eqref{1dflux}.

Consider then $d\geq 2$. We begin by introducing some notation. Let 
\begin{equation}\label{box}
\Lambda_k:=\{(x_1,\ldots,x_d)\in \ZZ^d: 1\leq x_j \leq k \mbox{ for all }j\leq d \},
\end{equation}
and denote by $U_A$ the uniform distribution on the finite set $A$, that is, the measure that assigns mass $|A|^{-1}$ to every point of $A$. Our goal is to connect $U_{\Lambda_{\ell}}$ to $ U_{\Lambda_1} $.

\begin{lemma}\label{threedflux}
	Let $k\in \{2,\ldots, \ell\}$. There exists a mass flow $\phi_k$ with support in the nearest-neighbour edges of $\Lambda_k$ such that 
	\begin{enumerate}
		\item $\phi_k: U_{\Lambda_k}\mapsto U_{\Lambda_{k-1}}$;
		\item $\phi_k\leq d \lrp{\frac{2}{k}}^d$. 
	\end{enumerate}
\end{lemma}

Before we prove the lemma, let us use it to prove Theorem \ref{flux_lemma}. Notice that the mass flow defined by

\begin{equation*}
\phi:=\sum_{k=2}^{\ell} \phi_k,
\end{equation*}
connects $ U_{\Lambda_{\ell}} $ to the point mass at $ (1,\ldots,1) $ (this can be checked directly from Definition \ref{flowdef}).

It remains to estimate the norm of $\phi$. Take a nearest-neighbour edge in $\Lambda_{\ell}$, say $(x,x-e_i)$, where $x\in \Lambda_k\setminus \Lambda_{k-1}$, $i\leq d$ and $k\leq \ell$. Notice that if  $j < k$ then $\phi_j(x,x-e_i)=0$ . Therefore

\begin{equation}
|\phi(x,x-e_i)|\leq \sum_{j=k}^{\ell}|\phi_j(x,x-e_i)|\leq \sum_{j=k}^{\ell}\frac{d 2^d}{j^d}\leq \frac{d2^d}{d-1}\frac{1}{(k-1)^{d-1}}.
\end{equation}
(the second inequality used Lemma \ref{threedflux}).

Since there are less than $k^{d-1}$ points in $\Lambda_k \setminus \Lambda_{k-1}$, 

\begin{equation}
\norm{\phi}^2 \leq c_d\sum_{k=2}^\ell k^{d-1}\lrp{\frac{1}{k^{d-1}}}^2,
\end{equation}
for $ c_d = 2^{1+d}/(d-1) $. This expression is of order $\log \ell$ when $d=2$ and order $1$ when $d\geq 3$.

\bigskip

\noindent\emph{Proof of Lemma \ref{threedflux}: }
For each $ j\in \{0,1,\ldots, d\} $, let $A_j$ be the set of those $(x_1,\ldots,x_d)\in \Lambda_k$ for which exactly $j$ entries are equal to $k$. Thus, $A_d$ is the corner $(k,\ldots,k)$; $A_{d-1}$ consists of $d$ line segments of length $k-1$; $A_{d-2}$ consists of $\binom{d}{2}$ squares of side length $k-1$, and so on. The $ A_j $ are pairwise disjoint, $ A_0 = \Lambda_{k-1} $ and $\bigcup_{j=1}^{d} A_j=\Lambda_k\setminus \Lambda_{k-1} $.

For each $ j\in \{0,1,\ldots, d\} $, let $ m_j:=U_{\Lambda_{k}}(A_j) $. 
Our strategy is to build flows $ \psi_d, \psi_{d-1}, \ldots, \psi_{1} $ whose supports are pairwise disjoint and such that 
\begin{equation}\label{key}
\psi_j: (m_d+\cdots + m_{d-j}) U_{A_j}\mapsto (m_d+\cdots + m_{d-j-1})U_{A_{j-1}}
\end{equation}
and $ |\psi_j|\leq 2^d k^{-d} $ for all $ j\in \{1,\ldots, d \} $. The lemma is then proved by taking $ \phi_k = \psi_d+\cdots + \psi_{1} $.

It is helpful to think of this construction as evolving in time. First, $ A_d $ spreads its mass uniformly along $ A_{d-1} $. Then $ A_{d-1} $ spreads its mass (plus the amount it got from $ A_d $) across $ A_{d-2} $. Then $ A_{d-2} $ spreads its mass (plus the amount it got from $ A_{d-1} $) uniformly across $ A_{d-3} $, and so on. 

Let $ x\in A_j $ and $ m=(m_0+\cdots +m_j)|A_j|^{-1} $ its mass at step $ j $. Notice that $ m\leq 2^dk^{-d} $.
Then $ x $ has exactly $ j $ coordinates equal to $ k $. It is adjacent to $ j $ line segments of $ A_{j+1} $. Using the one-dimensional flux \eqref{1dflux}, we can spread mass $ m/j $ at $ x $ uniformly along each of these segments. Call $ \psi^x_j $ the superposition of these $ j $ point-to-line flows. Notice that the $ \{\psi^x_j:x\in A_j\} $ have disjoint supports and that $ \psi^x_j \leq m \leq 2^dk^{-d} $. We can define $ \psi_j:=\sum_{x\in A_j}\psi^x_j $.
\hfill\qed

\begin{corollary}\label{fluxo_piramidal}
	Let $ \ell \in \{1,2,\ldots, n\} $. Recall the definition of $ \Lambda_{\ell} $ in \eqref{box}. Let $ p^{\ell}: \ZZ^d_n \to [0,1] $ be the uniform distribution in $ \Lambda_{\ell} $,  
	
	\begin{equation}\label{key}
	p^{\ell}(y)=\ell^{-d}\Ind{y\in \Lambda_{\ell}}.
	\end{equation}
	
	Define, for $ x\in\ZZ^d $, the ``pyramid kernel''
	
	\begin{equation}\label{key}
	q^{\ell}(x)=\sum_{y\in\ZZ^d}p^{\ell}(y)p^{\ell}(x-y).
	\end{equation}
	
	Then there exists a mass flow 
	\begin{equation}\label{key}
	\psi^{\ell}:\delta_0 \mapsto q^{\ell}
	\end{equation}
	with support in $ \Lambda_{2\ell+1} $ and $ \norm{\psi^{\ell}}^2 \leq\norm{\phi^{\ell}}^2$, where $ \phi^{\ell}:\delta_0\mapsto p^{\ell} $ is the flow constructed in Theorem \ref{flux_lemma}.
\end{corollary}

\begin{proof}
	One can take, for $ x\in \ZZ^d $ and $ j\in \{1,\ldots, d \} $. 
	\begin{equation}\label{key}
	\psi^{\ell}_{x,x+e_j}: = \phi^{\ell}_{x,x+e_j}+\sum_{y\in \ZZ^d} p^{\ell}(y)\phi^{\ell}_{x-y,x-y+e_j}.
	\end{equation}
\end{proof}

\section{Entropy estimate}

The present section is devoted to the proof of Theorem \ref{entropy}. The first two subsections establish several estimates that are also going to be needed in the proof of the Boltzmann-Gibbs Principle (Section \ref{section_bg}). In the diagram below we summarize the main implications in the proof of the entropy bound:

\bigskip

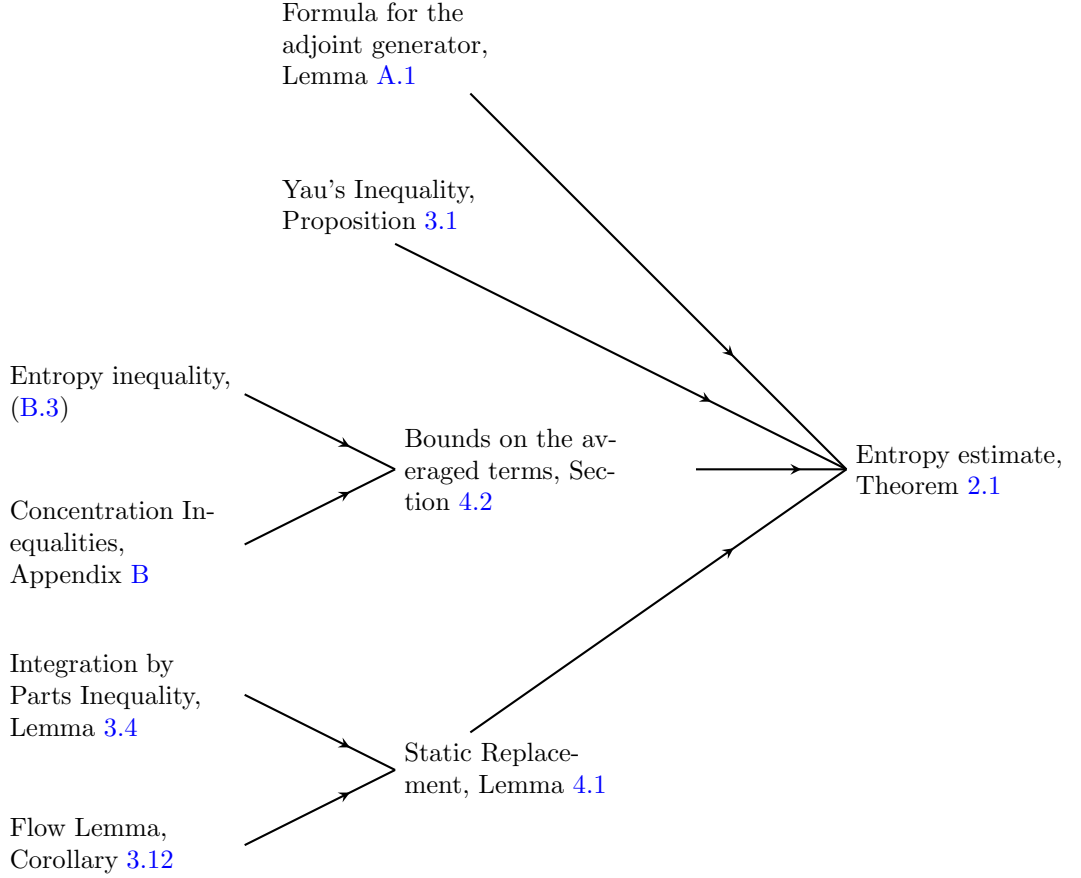
\begin{tikzpicture}[xshift=-5cm]
\draw [directed,thick] (0,0) --(2,1);  
\draw [directed,thick] (0,2) --(2,1);   
\draw [directed,thick] (0,4) --(2,5); 
\draw [directed,thick] (0,6) --(2,5);
\draw [directed,thick] (3,10) --(8,5);
\draw [directed,thick] (2,8) --(8,5);
\draw [directed,thick] (6,5) --(8,5);
\draw [directed,thick] (3,1.5) --(8,5);
\node [left, text width = 3cm] at (0,0) {Flow Lemma, Corollary \ref{fluxo_piramidal} };
\node [left, text width = 3cm] at (0,2) {Integration by Parts Inequality, Lemma \ref{integration_by_parts} };
\node [left, text width = 3cm] at (0,4) {Concentration Inequalities, \\ Appendix \ref{appendix_concentration}};
\node [left, text width = 3cm] at (0,6) {Entropy inequality, \eqref{entropy1}};
\node [right, text width = 3cm] at (2,1) {Static Replacement, Lemma \ref{entropy_replacement_statement}};
\node [right, text width = 3cm] at (2,5) {Bounds on the averaged terms, Section \ref{section_entropy_concentration}};
\node [above, text width = 3cm] at (2,8) {Yau's Inequality, Proposition \ref{yaus_ineq}};
\node [above, text width = 3cm] at (2,10) {Formula for the adjoint generator, Lemma \ref{ln*1}};
\node[right, text width = 3cm] at (8,5) {Entropy estimate, Theorem \ref{entropy}};
\end{tikzpicture}\\[.3cm]

We say that a function $ g:\{0,1\}^{\TT^d_n}\to \RR $ is a $ \nu_\rho $-\emph{density} if it is non-negative and $ \int g\,\mathrm{d}\nu_{\rho} =1 $.  Given a function $ \phi:\ZZ^d\to \RR $, we denote by $ \tilde{\phi} $ the reflection of $ \phi $ with respect to the origin, $ \tilde{\phi}(z):=\phi(-z) $.

\subsection{Static replacement}

Recall the definitions of the kernels $ p^{\ell} $ (uniform) and $ q^{\ell} $ (pyramidal) from Corollary \ref{fluxo_piramidal}. Here we replace the variable $ \etta_{x}:=\etta_x - \rho $ by its average with respect to $ q^{\ell} $ in a box to the right of $ x $. More precisely, given $ \xi:\{0,1\}^{\TT^d_n}\to \RR $, denote

\begin{equation}\label{key}
\lrp{\xi \star \tilde{q}^{\ell}}_x := \sum_{y\in \ZZ^d}  \xi_{x+y}q^{\ell}_y.
\end{equation}

The reason for taking averages with respect to 
$ q^{\ell} $ instead of $ p^{\ell} $ is that $ q^{\ell} $ is the algebraic identity \eqref{adjoint_convolution}, that will be crucial in the proof of Lemma \ref{conv_caixa} below.

\begin{lemma}\label{entropy_replacement_statement}
Let $\psi^{\ell}:0\mapsto q^{\ell}$ be the mass flow from Corollary \ref{fluxo_piramidal}.   Let $(h_x)_{x\in\TT_n^d}$ be a family of local functions.  Then, for any $ a>0 $,
	
\begin{multline}\label{fundamental_lemma_statement}
\int f\cdot\sum_{x\in\TT^d_n}h_x \lrp{\etta - \etta \star \tilde{q}^{\ell}}_x \,\mathrm{d}\nu_{\rho}
\leq a\int \Gamma_n\lrp{\sqrt f}\,\mathrm{d}\nu_{\rho}
\\
+ \frac{d}{a n^2}
\int f\cdot \sum_{j=1}^d\sum_{z\in\TT^d_n}
\lrp{\sum_{y\in\TT^d_n}h_{z-y}\psi^{\ell}_{y,y+e_j}
	}^2\,\mathrm{d}\nu_{\rho},
\end{multline}
	under the assumption that the support of $ h_x $ does not intersect $ x+\Lambda_{2\ell +1} $.

\end{lemma}

\begin{proof}
We start with the telescoping identity 
\begin{equation}\label{key}
\etta_x - \lrp{\etta \star \tilde{q}^\ell}_x = \sum_{j=1}^d\sum_{y\in\ZZ^d}\psi^\ell_{y,y+e_j}\lrp{ \etta_{x+y} - \etta_{x+y+e_j} }.
\end{equation}
The lefthand side of \eqref{fundamental_lemma_statement} can then be written as 
\begin{equation}\label{key}
\sum_{j=1}^d 
\int f\cdot
\sum_{z\in\TT^d_n}\lrp{\etta_z - \etta_{z+e_j}}
\lrp{ \sum_{y\in \TT^d_n} h_{z-y} \psi^\ell_{y,y+e_j}  }
\,\mathrm{d}\nu_{\rho}.
\end{equation}
To finish the proof, we apply the Integration by Parts Lemma \ref{integration_by_parts} to each term in the sum.

\end{proof}

\subsection{Concentration estimates}\label{section_entropy_concentration}
Let $ \ell, \ell_0 \in \{1,\ldots, n \} $ be fixed.  Denote by $ \Lambda_{\ell} $ the box of size $ \ell $ to the right of the origin, $ \Lambda_{\ell}:=\{0,\ldots, \ell -1 \}^d $. Given a finite set $ A \subset -\Lambda_{\ell_0} $. denote
\begin{equation}\label{key}
\etta_A:=\prod_{x\in A}(\eta_x - \rho).
\end{equation}
 
 Throughout this section, we use the same notation as in Corollary \ref{fluxo_piramidal}: $ p^{\ell} $ denotes the uniform distribution in $ \Lambda_{\ell} $, $ q^{\ell} $ denotes the convolution of $ p^{\ell} $ with itself and $ \psi^{\ell} $ denotes the mass flow from Corollary \ref{fluxo_piramidal}, that connects $ q^{\ell} $ to the point mass at the origin at a cost of order $ g_d(\ell) $.

\begin{lemma}\label{conv_fluxo}
There exists a positive $ C =C(d,A)$ such that
	\begin{equation}
	\int f\cdot  \sum_{x\in \TT^d_n} \lrp{\sum_{y\in \TT^d_n} \psi^{\ell}_{y,y+e_j}\etta_{A+x-y} }^2\,\mathrm{d}\nu_{\rho} \leq C \ell^d g_d(\ell)
	\lrp{ H(f) + 
		\frac{n^d}{\ell^d}
	}.
	\end{equation}
\end{lemma}

\begin{proof}
	Let $ \gamma >0 $. By the entropy inequality \eqref{entropy1}, the lefthand side is bounded by 
	\begin{equation}\label{lemma42_a}
	\frac{H(f)}{\gamma} + \frac{1}{\gamma}\log \int \exp \lrch{
		 \gamma\sum_{x\in \TT^d_n} \lrp{\sum_{y\in \TT^d_n} \psi^{\ell}_{y,y+e_j}\etta_{A+x-y} }^2
}
\,\mathrm{d}\nu_{\rho}.
	\end{equation}
Now we try to find the largest possible $ \gamma $ for which the last exponential moment is finite. The guiding idea behind the proof is that the variables $ \{\etta_{A+x}:x\in \TT^d_n\} $ concentrate around their means like independent subgaussian random variables of parameter $ 1 $ and the convolution inside the square is like a $ \norm{\psi^\ell}^2 $-subgaussian variable. We need to be careful with the dependencies, though.  We will circunvent them by applying H\"older's inequality a couple times and keeping track of the errors. It turns out the largest possible $ \gamma $ is of order $ \ell^{-d} \norm{\psi^\ell}^{-2} $. 
The factor $ \ell^d $ comes from the dependency of range $ \ell $ between the variables $ \lrch{\sum_{y\in \TT^d_n} \psi^{\ell}_{y,y+e_j}\etta_{A+x-y}: x \in \TT^d_n} $.

Let us begin by denoting $ \xi_x:=\sum_{y\in \TT^d_n} \psi^{\ell}_{y,y+e_j}\etta_{A+x-y} $. We only need two properties of $ \xi $. The first is that, under the measure $ \nu_{\rho} $, the variables $ \xi_x $ and $ \xi_y $ are independent whenever $ |x-y| > 2\ell + 1 + \ell_0$. The second is that each $ \xi_x $ is a subgaussian random variable of parameter $ \kappa_A \norm{\psi^\ell}^2 $, where $ \kappa_A $ depends only on $ A $ and $ d $, not on $ \ell $ or $ d $.

Assume both properties for a moment. It is possible to show (see Lemma \ref{torus_coloring}) that there exists a partition of the torus $ \TT^d_n = \sqcup_{i\in \mathcal I} B_i$ with the property that, for each $ i \in \mathcal I $, the family $ \{\xi_x:x\in B_i\} $ is independent under $ \nu_\rho $. In addition, $ |\mathcal I|\leq \kappa \ell^d$, where the constant $ \kappa $ depends only on the diameter of $ A $ and on the dimension $ d $. One can then combine H\"older's inequality and independence to estimate
\begin{equation}\label{partition+holder}
\begin{aligned}
\frac{1}{\gamma}\log \int e^{
\gamma \sum_{x\in\TT^d_n} \xi_x^2
}\,\mathrm{d}\nu_{\rho} & = 
\frac{1}{\gamma}\log \int e^{
	\gamma \sum_{i\in\mathcal I} \sum_{x\in B_i}\xi_x^2
}\,\mathrm{d}\nu_{\rho}\\
& \leq  \frac{1}{\gamma \kappa \ell^d} \sum_{i\in\mathcal I}
\log \int e^{
	\gamma \kappa \ell^d \sum_{x\in B_i}\xi_x^2
}\,\mathrm{d}\nu_{\rho}\\
&= \frac{1}{\gamma \kappa \ell^d} \sum_{x\in \TT^d_n}
\log \int e^{
	\gamma \kappa \ell^d \xi_x^2
}\,\mathrm{d}\nu_{\rho}.
\end{aligned}
\end{equation}

Now, an application of inequality \eqref{chisquare_expmom} tells us that, if $ \gamma $ is small enough to ensure $ \gamma \kappa \ell^d \leq \frac{1}{4\kappa_A\norm{\psi^\ell}^2} $, then the last term is bounded by $ 8\kappa_A \norm{\psi^\ell}^2 n^d $. To finish the proof, we can choose $ \gamma = \frac{1}{4\kappa_A\ell^d\norm{\psi^\ell}^2} $ and plug the last bound in \eqref{lemma42_a}. 

The only thing left to check is the subgaussianity of the variables $ \{\xi_z: z\in \TT^d_n\} $. To estimate the exponential moments of $ \xi_z $, one can combine Lemma \eqref{torus_coloring} and computation \eqref{partition+holder}, with $ \etta_{A+z-x}\psi^\ell_{x,x+e_j} $ in place of $ \xi_x^2 $.

\end{proof}

\begin{lemma}\label{conv_caixa}
There exists a positive $ C = C(d,A) $ such that 
	\begin{equation}\label{key}
	\int f\cdot \lvert  \sum_{x\in\TT^d_n}\etta_{A+x}\lrp{\etta \star \tilde{q}^{\ell} }_x  \rvert \,\mathrm{d}\nu_{\rho}
	\leq C \lrp{
	H(f) +  \frac{n^d}{\ell^d}}
	.
	\end{equation}
\end{lemma}

\begin{proof}
	The proof is very similar to the proof of Lemma \ref{conv_fluxo}, but now the subgaussian random variables have variance of order $ \ell^{-d} $ instead of $ \norm{\psi^\ell}^2 $, and this is reflected in the upper bound. We start the proof by fixing $ \gamma > 0$ and applying \eqref{entropy1}, bounding the lefthand side by
	\begin{equation}\label{lemma43_b}
	\frac{H(f)}{\gamma}+\frac{1}{\gamma}\log \int \exp\lrch{
	\gamma\lvert \sum_{x\in\TT^d_n}\etta_{A+x}\lrp{\etta \star \tilde{q}^{\ell} }_x \rvert
}\,\mathrm{d}\nu_{\rho}.
	\end{equation}

Now we try to find the largest $ \gamma $ such that the logarithm above is bounded by a constant that does not depend on $ n $ nor $ \ell $.  To begin, we need the following identity, which is the reason why we chose to take averages with respect to $ q^\ell $ instead of $ p^\ell $:
\begin{equation}\label{adjoint_convolution}
\sum_{x\in\TT^d_n} \etta_{A+x}\lrp{ \etta \star \tilde{q}^\ell }_x =  \sum_{x\in \TT^d_n} \lrp{\etta_A \star p^\ell}_x\cdot \lrp{\etta \star \tilde{p}^\ell }_x.
\end{equation}	
In words, the adjoint of convolution with some kernel is convolution with the reflected kernel. Denoting $ \xi_x:=|\lrp{\etta_A \star p^\ell}_x\cdot \lrp{\etta \star \tilde{p}^\ell }_x |$, we can repeat computation \eqref{partition+holder} (with $ \xi_x $ in place of $ \xi_x^2 $) and get 

\begin{equation}\label{lemma43_c}
\begin{aligned}
\frac{1}{\gamma}\log \int e^{
	\gamma\lvert \sum_{x\in\TT^d_n}\etta_{A+x}\lrp{\etta \star \tilde{q}^{\ell} }_x \rvert
}\,\mathrm{d}\nu_{\rho} 
&\leq \frac{1}{\gamma \kappa \ell^d}
\sum_{x\in \TT^d_n}
\log \int e^{
	\gamma \kappa \ell^d |\lrp{\etta_A \star p^\ell}_x\cdot \lrp{\etta \star \tilde{p}^\ell }_x|
}\,\mathrm{d}\nu_{\rho}.
\end{aligned}
\end{equation}

In the above equation and in the remaining of the proof, $ \kappa $ denotes a positive number that depends only on $ \ell_0 $ and $ d $, not on $ \ell $ nor $ n $. 
Applying inequality $ |ab|\leq \frac{a^2}{2}+ \frac{b^2}{2} $ to the exponent and then Cauchy-Schwartz inequality to the integral, it is possible to bound the righthand side by

\begin{equation}\label{lemma43_d}
\frac{1}{2\gamma \kappa \ell^d}
\sum_{x\in \TT^d_n}
\log \int e^{
	\gamma \kappa \ell^d \lrp{\etta_A \star p^\ell}^2_x
}\,\mathrm{d}\nu_{\rho}
+
\frac{1}{2\gamma \kappa \ell^d}
\sum_{x\in \TT^d_n}
\log \int e^{
	\gamma \kappa \ell^d \lrp{\etta \star \tilde{p}^\ell}^2_x
}\,\mathrm{d}\nu_{\rho}.
\end{equation}

Notice that, under $ \nu_{\rho} $, each convolution $ \lrp{\etta\star \tilde{p}^\ell}_x $ is a subgaussian random variable of parameter $ \sum_{y\in \ZZ^d}\lrp{p^{\ell}_y}^2 =\ell^{-d}  $. Applying \eqref{chisquare_expmom}, we can bound the second term of \eqref{lemma43_d} by $ \frac{4n^d}{\ell^d} $ provided $ \gamma \leq \frac{1}{4\kappa} $.

It remains to bound the first term in \eqref{lemma43_d}. The computation is essentially the same as for the second term, but now the variables $ \lrch{ \lrp{\etta_A \star p^\ell}_x:x\in \TT^d_n } $ have a depencency of range $ \ell_0 $. It is possible to prove that $\lrp{\etta_A \star p^\ell}_x$ is subgaussian with parameter $ \kappa_A\ell^{-d}$, where $ \kappa_A $ depends only on $ \ell_0 $ and $ d $. We can then apply \eqref{chisquare_expmom} and bound the first term of \eqref{lemma43_d} by
$ 4\kappa_A\frac{n^d }{\ell^d} $ provided that $ \gamma \leq \frac{1}{4\kappa_A } $. To finish the proof, we go back to \eqref{lemma43_b} and substitute $ \gamma = \frac{1}{4(\kappa + \kappa_A)}  $.
\end{proof}

\begin{remark}
	It is also possible to estimate the exponential moments of the variables $ \etta_{A+x}\lrp{\etta \star \tilde{q}^\ell}_x $ directly, avoiding the use of identity \eqref{adjoint_convolution}. The trick is to condition on the value of $ \etta_{A+x} $ and to exploit its boundedness. We chose to present the above proof because we found the computation instructive. Besides, identity \eqref{adjoint_convolution} is indispensable in situations where the reference measure is allowed to change in time. 
\end{remark}

\subsection{Proof of Theorem \ref{entropy}}  

Denote $ H_n(t):= H(\mu^n_t|\nu_{\rho}) $.
We claim that there exists $ C > 0 $ that does not depend on $ \ell $ nor on $ n $ such that
\footnote{The constant $ C $ depends on the model though, through the coefficients of $ L_n^*\ind{} $ and the number of terms in its expression.}

\begin{equation}\label{entropy_production_bound}
\partial_tH_n(t)\leq 
 C\lrp{1+C\frac{\ell^d g_d(\ell)  }{n^2}} 
\lrp{ H_n(t) + \frac{n^d}{\ell^d} }.
\end{equation}

	Let us finish the proof assuming the last inequality. To set up an application of Gronwall's inequality, choose $ \ell \in \{1,\ldots, n\} $ such that $ \ell^d\norm{\phi^\ell}^2 \leq n^2 $. A possible choice is 
	\begin{equation}\label{choosing_box_size}
	\ell = 
	\left\{
	\begin{array}{rl}
	n, & d=1,\\
	\frac{n}{\sqrt {\log n}}, & d=2,\\
	n^{2/d}, & d\geq 3.
	\end{array}
	\right.
	\end{equation}
With the choices above and the assumption that the entropy at time zero is null, an application of Gronwall's inequality yields $ H_n(t) \leq \frac{n^d}{\ell^d} e^{Ct} $, finishing the proof.

Now it remains to prove \eqref{entropy_production_bound}. We start with Yau's Inequality \ref{yaus_ineq}: if $ f^n_t $ is the Radon-Nykodym density of the law of $\eta^n_t $ with respect to $ \nu_{\rho} $ then 

\begin{equation}\label{bound_for_terms_of_l*1}
\partial_t H_n(t)\leq \int L_n^*\ind{} \cdot f^n_t -\Gamma_n\lrp{\sqrt{ f^n_t}}\,\mathrm{d}\nu_{\rho}, 
\end{equation}
where $ L_n^* $ denotes the adjoint of $ L_n $ in $ L^2(\nu_{\rho}) $ and $ \Gamma_n $ denotes the  carr\'{e} du champ operator associated to $ L_n $. We chose $ \rho $ in such a way that $ L_n^* $ is a polynomial in the variables $ \{\etta_x:=\eta_x - \rho:x\in \TT_n^d\} $ of order bigger than $ 1 $, see Proposition \ref{ln*1}.  We are going to prove that the integral 
against a $ \nu_{\rho} $-density $ f:\{0,1\}^{\TT_n^d}\to \RR_+ $
of each of the terms in the expression for $ L_n^*1 $ is bounded by

\begin{equation}\label{fundamental_bound}
a\int \Gamma_n\lrp{\sqrt f}\,\mathrm{d}\nu_{\rho} + 
C\lrp{1+\frac{\ell^d g_d(\ell) }{an^2}}
\lrp{ H(f) + \frac{n^d}{\ell^d} },
\end{equation}
where $ a>0 $ is arbitrary and $ C $ does not depend on $ n $ nor on $ \ell $. We need the freedom in the choice of $ a $ so that we can sum the bounds for each term in $ L_n^*1 $ to cancel the carr\'e du champ in Yau's Inequality. 

From now on we don't need any more input from the model. 
In the remaining of the proof put together the inequalities of the present section to bound the integral $ \int f\cdot\sum_{x\in\TT^d_n}\etta_{x-2e_1}\etta_{x-e_1}\etta_{x}\,\mathrm{d}\nu_{\rho} $ by the expression in \eqref{fundamental_bound}. The proofs for the other terms in $ L_n^*\ind{} $ differ only in notation.

The first step is to replace the rightmost variable, $ \etta_{x} $ by its average $\lrp {\etta\star \tilde{q}^{\ell}}_{x} $.  Applying Lemma \ref{entropy_replacement_statement}, we get

\begin{equation}
\int f\cdot \sum_{x\in\TT^d_n}
	\etta_{x-2e_1}\etta_{x-e_1} \lrp{\etta - \etta\star \tilde{q}^{\ell}}_{x}
		\,\mathrm{d}\nu_{\rho}
\leq a \int \Gamma_n\lrp{ \sqrt f}\,\mathrm{d}\nu_{\rho} 
+ \frac{d}{an^2}\int f\cdot W^{\ell}\,\mathrm{d}\nu_{\rho},
\end{equation}
where 
\begin{equation}
\begin{aligned}
W^{\ell}(\eta) 
&= \sum_{j=1}^d\sum_{z\in\TT^d_n}
\lrp{\sum_{y\in \TT_n^d} \etta_{z-y-2e_1}\etta_{z-y-e_1}
	\psi^{\ell}_{y,y+e_j}}^2.
\end{aligned}
\end{equation}

Applying Lemma \ref{conv_fluxo} we find that, for $ C>0 $ large enough,

\begin{equation}\label{key}
\frac{d}{an^2 }\int f\cdot W^{\ell}\,\mathrm{d}\nu_{\rho} \leq  \frac{C\ell^d g_d(\ell) }{an^2}\lrp{ H(f)+ \frac{n^d}{\ell^d}  }.
\end{equation}

Combining the last inequality with that from Lemma \ref{conv_caixa}, we discover that each term in the expression for $ L_n^*1 $ is bounded by \eqref{fundamental_bound}, and this
finishes the proof of Theorem \ref{entropy}.

\hfill\qed

\section{Boltzmann-Gibbs Principle}\label{section_bg}

We use the notation that was introduced at the beginning of Section \ref{section_entropy_concentration}.
\begin{proposition}\label{bg}
	Let $ d\in \{1,2,3\} $. Fix positive numbers $ \delta $  and $ t $ and a nonempty finite set $ A_0 \subset \{(z_1,\ldots, z_d)\in \ZZ^d: z_j<0 \mbox{ for all } j  \} $. Let $ \varphi: \TT^d \to \RR $ be a smooth function. Then 
	  \begin{equation}\label{key}
	  \lim_{n\to \infty}\PP_{\nu_{\rho}} \lrp{\Big|
 \int_0^t \sum_{x\in \TT^d_n} n^{-d/2}\varphi\lrp{\frac{x}{n}}\,\etta_{A_0+x}(s) \etta_x(s) \,\mathrm{d}s	  
  \Big| > \delta } =0.
	  \end{equation}
\end{proposition}

\begin{proof}
	Let $ \ell \in \{1,\ldots, n\} $. Define the functions $ V, V^\ell:\{0,1\}^{\TT^d_n}\to \RR $ by
\begin{equation}\label{v}
V(\eta) = \sum_{x\in \TT^d_n}\varphi\lrp{\frac{x}{n}}\,\etta_{A_0+x} \etta_x
\end{equation}
and
\begin{equation}\label{vl}
V^\ell(\eta) = \sum_{x\in \TT^d_n}\varphi\lrp{\frac{x}{n}}\,\etta_{A_0+x} \lrp{\etta \star \tilde{q}^\ell}_x.
\end{equation}

The proof is based on the Feynman-Kac Inequality, Lemma \ref{fkineq}.  
During this proof, we are going to omit the dependency in $ \eta(s) $ in the time integrals, to make the formulas cleaner. We are also going to omit the dependency of the integrands in $ n $ and $ s $.

In order to apply Lemma \ref{fkineq}, we need to get rid of the absolute value inside the probability.
We are going to find functions $ M_+, M_-: \{0,1\}^{\TT^d_n} \to \RR $ such that, for any positive $ \delta $
\begin{equation}
\lim_{n\to \infty}\PP_{\nu_{\rho}}\lrp{
	\int_0^t \pm n^{-d/2}V - M_{\pm} \,\mathrm{d}s	   > \delta 
} = 0
\end{equation}
and 
\begin{equation}
\lim_{n\to \infty}\PP_{\nu_{\rho}}\lrp{\Big|
	\int_0^t M_{\pm}  \,\mathrm{d}s \Big|	  > \delta 
} = 0.
\end{equation}
If we are able to find such functions, the proof is finished. We are going to find only $ M_+ $, for $ M_- $ can be found in the same way.

Let us apply Lemma \ref{fkineq} to the integral in \eqref{tira}. For any positive $ \theta $, this lemma implies

\begin{multline}\label{bg_eq01}
\log \PP_{\nu_{\rho}} \lrp{ \int_0^t V - M_+ \,\mathrm{d}s > \delta} \leq -\theta \delta +\\
+ \sup_{f} \lrch{
\int f\cdot \theta \lrp{V - M_+} \,\mathrm{d}\nu_{\rho} + \frac{1}{2}\int f\cdot L_n^*\ind{} \,\mathrm{d}\nu_{\rho} - \frac{1}{2}\Gamma_n\lrp{\sqrt f}
},
\end{multline}
where the supremum runs over all $ \nu_{\rho} $-densities $ f :\{0,1\}^{\TT_n^d} \to \RR_+$. 

The plan is to apply Lemma \ref{entropy_replacement_statement} to bound some terms by the carr\'e du champ, and then to choose $ M_+ $ that cancels all the error terms. Before this lemma can be applied, we need to set up some notation.

Recall expression \eqref{ln*1} for $ L_n^*\ind{} $. Let us write it in the form

\begin{equation}\label{key}
L_n^*\ind{}(\eta) = \sum_{A \in \mathcal A}\sum_{x\in\TT^d_n}G_x\etta_{A+x}\etta_x,
\end{equation}
where $ \mathcal A = \cup_{j=1}^d \{ \{-2e_j, -e_j\}, \{-2e_j\}, \{-e_j\}  \} $ is a finite family of nonempty subsets of $ \ZZ^d $ and $ \{ G_x:x\in \TT^d_n\} $ are uniformly bounded constants. We chose to write $ L_n^*\ind{} $ in this format in order to make it easier to apply the inequalities of Section 4 and to show that the proof does not depend on the specifics of the model. The only thing that is necessary is that $ L_n^*\ind{} $ have degree at least two in the chosen reference measure.

Denote $ W:= L_n^*\ind{} $ and 
\begin{equation}\label{key}
W^\ell (\eta)= \sum_{A \in \mathcal A}\sum_{x\in\TT^d_n}G_x\etta_{A+x}\lrp{\etta \star \tilde{q}^\ell}_x.
\end{equation}

Applying Lemma \ref{entropy_replacement_statement} to $ V $ and to each of the terms in $ W $, we get the inequalities
	\begin{equation}\label{bg_byparts_w}
\int  \frac{1}{2}f\lrp{W-W^{\ell}} \,\mathrm{d}\nu_{\rho} \leq \frac{1}{4}\int \Gamma_n\lrp{ \sqrt f} \,\mathrm{d}\nu_{\rho} + \frac{d}{n^2}\int f\cdot W_*^{\ell}\,\mathrm{d}\nu_{\rho}
\end{equation}
and
	\begin{equation}\label{bg_byparts_v}
	\int \theta n^{-d/2} f\lrp{V-V^{\ell}} \,\mathrm{d}\nu_{\rho} \leq \frac{1}{4}\int \Gamma_n\lrp{ \sqrt f} \,\mathrm{d}\nu_{\rho} + \frac{4d\theta^2}{n^{2+d}}\int f\cdot V_*^{\ell}\,\mathrm{d}\nu_{\rho},
	\end{equation}
where 
\begin{equation}\label{key}
V_*^\ell(\eta)  = \sum_{j=1}^d\sum_{x\in\TT^d_n}
\lrp{\sum_{y\in\TT^d_n}\varphi\lrp{\frac{x-y}{n}}\,\etta_{A_0+x-y}\psi^{\ell}_{y,y+e_j}
}^2 
\end{equation}
and
\begin{equation}
W_*^\ell (\eta) =  \sum_{A \in \mathcal A}\sum_{j=1}^d\sum_{x\in\TT^d_n}\lrp{\sum_{y\in\TT^d_n} G_{x-y}\etta_{A+x-y}\psi^\ell_{y,y+e_j} }^2
\end{equation}	

Going back to \eqref{bg_eq01}, we see that, for 

\begin{equation}\label{def_M}
M_+(\ell,\theta) :=  \frac{1}{n^{d/2}}V^{\ell} + \frac{4d\theta}{n^{2+d}}V^{\ell}_* + \frac{1}{2\theta}W^{\ell} + \frac{d}{\theta n^2}W^\ell_*,
\end{equation}
the inequality
 	\begin{equation}\label{tira}
	\log \PP_{\nu_{\rho}}\lrp{ 
	\int_0^t \frac{1}{n^{d/2}}V - M_+(\ell,\theta) \,\mathrm{d}s
 >\delta} \leq -\theta\delta
	\end{equation}
holds for any positive $ \theta $ and $ \delta $. It remains to check whether there is some choice of $ \theta \gg 1 $ and $ \ell \in \{1,\ldots, n\} $ such that 
\begin{equation}\label{bota}
\lim_{n\to \infty}\PP_{\nu_{\rho}}\lrp{\Big|
	\int_0^t M_{+}(\ell, \theta)\,\mathrm{d}s  \Big|	   > \delta 
} = 0.
\end{equation}
To bound this time integral we need the entropy bound in Theorem \ref{entropy}. 

\begin{lemma}\label{bg_averaged}
	There exists a positive $ C = C(\mathcal A, \norm{G}_{\infty})$  such that the following inequalities hold for any choice of $ \delta >0 $:
	\begin{equation}\label{vtail}
	\PP_{\nu_{\rho}}\lrp{ \Big|
\int_0^t \frac{1}{n^{d/2}}V^\ell \,\mathrm{d}s
\Big| >\delta } \leq \frac{Ct}{\delta}\frac{1}{n^{d/2}} \lrp{ n^{d-2} g_d(n) + \frac{n^d}{\ell^d}  } 
	\end{equation}
	\begin{equation}\label{wltail}
	\PP_{\nu_{\rho}}\lrp{ \Big|
		\int_0^t \frac{1}{\theta}W^{\ell}
		\,\mathrm{d}s
		\Big| >\delta } \leq \frac{Ct}{\delta}\cdot \frac{1}{\theta} \lrp{ n^{d-2} g_d(n) + \frac{n^{d}}{\ell^d}  }
	\end{equation}
		\begin{equation}\label{vltail}
	\PP_{\nu_{\rho}}\lrp{ \Big|
		\int_0^t \frac{\theta}{n^{2+d}}V^{\ell}_*
		\,\mathrm{d}s
		\Big| >\delta } \leq  \frac{Ct}{\delta}\cdot \frac{\theta}{n^d}\frac{\ell^{d} g_d(\ell)}{n^{2}}\lrp{n^{d-2} g_d(n) + \frac{n^{d}}{\ell^d}  }
	\end{equation}
		\begin{equation}\label{w*tail}
	\PP_{\nu_{\rho}}\lrp{ \Big|
		\int_0^t \frac{1}{\theta}\frac{1}{n^{2}}W^\ell_*
		\,\mathrm{d}s
		\Big| >\delta } \leq \frac{Ct}{\delta }\cdot \frac{1}{\theta}\frac{\ell^{d}g_d(\ell)}{n^2} \lrp{n^{d-2} g_d(n) + \frac{n^{d}}{\ell^d}  }.
	\end{equation}
\end{lemma}

\begin{proof}
For each $ s\leq t $, denote by $ \mu_s $ the law of the reaction-diffusion process at time $ s $.  An application of Markov's inequality bounds the lefthand side of \eqref{wltail} by $ \frac{1}{\delta \theta}\int_0^t \EE_{\mu_s}|W^{\ell}|\,\mathrm{d}s $.  Applying inequality \eqref{entropy1}, Theorem \ref{entropy} and Lemma \ref{conv_caixa} to each of the terms in $ W^\ell $, we obtain \eqref{wltail}.
Notice that the constant $ C $ in the righthand side of \eqref{wltail} is different from the constant in Lemma \ref{conv_caixa}, because it depends on $ |\mathcal A| $ and $ \norm{G}_{\infty}$.

The remaining inequalities can be proven in the same way. For the last two, Lemma \ref{conv_fluxo} is used in place of Lemma \ref{conv_caixa}.
\end{proof}

\bigskip

To bound \eqref{bota} we use Lemma \ref{bg_averaged}. It is enough to find $ \theta \gg 1$ and $ \ell \in \{1,\ldots, n \} $ such that each of the upper bounds in this lemma vanishes as $ n\to \infty $. We can begin with $ \ell $, repeating the choices in \eqref{choosing_box_size}. Recall that, for those choices, $ \ell^dg_d(\ell) \leq n^2 $ and $ n^{d-2}g_d(n) = \frac{n^d}{\ell^d}(1+o(1)) $. Once $ \ell $ is chosen, one can check that all the upper bounds in Lemma \ref{bg_averaged} vanish when 
\begin{equation}
\begin{aligned}
1 \ll \theta \ll n, &\quad d=1;\\
\log n \ll \theta \ll \frac{n^2}{\log n}, &\quad  d=2;\\
n \ll \theta \ll n^2, &\quad d=3.
\end{aligned}
\end{equation}
 Notice that when $ d \geq 4 $ there is no way of controlling \eqref{vtail}, so our proof works only in dimensions $ 1 $, $ 2 $ and $ 3 $. The aforementioned choices of $ \theta $ and $ \ell $ define $ M_+ $ in \eqref{def_M} and ensure the validity of inequalities \eqref{tira} and $ \eqref{bota} $, thus concluding the proof.
\end{proof}

\section{Tightness}\label{section_tightness}

The proof uses the Kolmogorov-Centov criterion, see Problem 2.4.11 in \cite{karatzas}.

\begin{proposition}\label{kolmogorov_centov}
	Assume that the sequence of stochastic processes $ \{Y^n_t:
	t\in [0,T]\} _{n\in\NN}$  satisfies 
	\begin{equation}
	\limsup_{n\to\infty}\EE[|Y^n_t - Y^n_s|^\theta]\leq C|t-s|^{1+\theta '}
	\end{equation}
	for some positive constants $ \theta $, $ \theta ' $ and $ C $ and for all $ s,t\in [0,T] $. Then it also satisfies
	\begin{equation*}
	\lim_{\delta\to 0}\limsup_{n\to\infty}\PP\left(
	\sup_{\substack{|t-s|\leq \delta \\ s,t \in [0,T]}}|Y^n_t - Y^n_s|>\varepsilon 
	\right)=0, \mbox{ for all } \varepsilon > 0.
	\end{equation*}
\end{proposition}

\bigskip

More precisely, we will prove the following:

\begin{theorem}
	Consider the reaction-diffusion process with generator \eqref{rdgenerator} in dimension $ 1 $, starting from the Bernoulli product measure $ \nu_{\rho} $, with $ \int c_x \,d\nu_{\rho} =0$ for all $ x\in \TT_n $. For any $ \theta > 1 $ and for any $ s,t\in [0,T] $, there exists a constant $ C = C(\theta, f) $ such that
	\begin{equation}
	\EE_{\nu_{\rho}}\left[ \Big |
	\int_s^t L_nX^n_r(f)\,\mathrm{d}r
	\Big |^{\theta} \right]
	\leq C(t-s)^{\theta}.
	\end{equation}
\end{theorem}
Tightness follows by choosing $ \theta > 1 $ and applying Proposition \ref{kolmogorov_centov}.

\begin{remark}
The main reason for our sticking to the one-dimensional case is that the tightness proof is much shorter. In dimensions $ 2 $ and $ 3 $ the proof we provide in the present section does not work. 
The only strategy we know to estimate the time integrals is to mimic the technique used in the proof of the Boltzmann-Gibbs principle.
\end{remark}

\begin{proof}
	
We start by estimating 
$\nu_{\rho}(L_nX^n(f)>\delta)$, and for that we 
use the Bounded Differences Inequality, 
Proposition \ref{bounded_differences}. Recall 
expression \eqref{formula_dynkin} for 
$L_nX^n_s(f)$. One can check 
\begin{equation}\label{key}
|L_nX^n(f)(\eta^x)-L_nX^n(f)(\eta)| \leq 3n^{-1/2}\norm{f}_{\infty}.
\end{equation}
Applying the Bounded Differences Inequality, we get
\begin{equation}
\log \nu_{\rho}(L_nX^n(f)>\delta) \leq -\frac{2\delta^2}{3\norm{f}_{\infty}^2}.
\end{equation} 
Recall from Theorem \ref{entropy} that the entropy is of order $1$. Plugging the last bound into the entropy inequality \eqref{entropy2} we find $K_f>0$ that depends only on $T$ and on $f$ such that, for all $t\in[0,T]$,
	
	\begin{equation}\label{ivysaur}
	\mathbb \mu^n_t\left(  |L_nX^n_r(f)|
	> \delta  \right) \leq \frac{K_f}{\delta^2}.
	\end{equation}
	
	Let $ \theta >1 $. Applying Lemma \ref{tail_lemma}, we get
	
	\begin{equation}
	\mathbb E_{\nu_{\rho}} [|L_nX^n_t|^{\theta}]\leq K_f^{\theta/2}\mbox{ for all }t\in [0,T].
	\end{equation}
	
	We finish the proof with an application of Jensen's inequality:
	
	\begin{equation}
	\begin{split}
	\mathbb E_{\nu_{\rho}} \left[\Big|\int_s^t L_n X^n_r(f)|\,\mathrm{d}r \Big|^{\theta}\right] &\leq 
	(t-s)^{\theta} \cdot \frac{1}{t-s}\int_s^t \mathbb E_{\nu_{\rho}}[|L_n X^n_r(f)|^{\theta}]\,\mathrm{d}s\\
	&\leq K_f^{\theta/2}\cdot (t-s)^{\theta}.
	\end{split}
	\end{equation}

\end{proof}

\appendix

\section{Computations involving the generator}\label{section_computations_generator}

\begin{lemma}\label{ln*1}
	Let $ L_n^* $ denote the adjoint of the generator $ L_n $ (defined in \eqref{rdgenerator}) in $ L^2(\nu_{\rho}) $. Then 
\begin{equation}\label{formula_ln*1}
	\begin{aligned}
		L_n^*\ind{}(\eta) = & 2\lambda \sum_{j=1}^d\sum_{x\in \TT^d_n} 
		\etta_{x-e_j}\etta_{x}\\
		&+  \frac{\lambda}{\rho}
		\sum_{j=1}^d\sum_{x\in \TT^d_n} 
		\etta_{x-2e_j}\,\etta_{x-e_j}\,\etta_{x}. 
	\end{aligned}
\end{equation}

\begin{proof}
	By reversibility, the exclusion part of the generator is self-adjoint in $ L^2(\nu_{\rho}) $. Therefore, we only need to deal with the birth-and-death part of the generator.
	Applying the explicit formula for the adjoint generator given in Proposition \ref{adjointformula}, we obtain
	\begin{equation}\label{key}
	\begin{aligned}
	L^*_n\ind{}(\eta) & = \sum_{x\in\TT^d_n} \eta_x \lrch{
	c_x^+(\eta)\frac{1-\rho}{\rho}  - c_x^-(\eta)
 }\\
&+ \sum_{x\in\TT^d_n} (1-\eta_x)\lrch{ c_x^-(\eta)\frac{\rho}{1-\rho} - c_x^+(\eta)
}\\
&= \sum_{x\in\TT^d_n} \lrp{ \frac{\eta_x}{\rho} - \frac{1-\eta_x}{1-\rho} } \lrch{ 
c_x^+(\eta) (1-\rho)  - c_x^-(\eta)\rho
}.
	\end{aligned}
	\end{equation}

We would like to write the above expression as a polynomial in the variables $\{ \etta_x := \eta_x - \rho: x\in \TT^d_n \} $. It can be seen from the above expression that the coefficient of the independent term is null. Besides, the assumption $ \int c_x\,\mathrm{d}\nu_\rho =0 $ implies that all terms of degree $ 1 $ vanish. Substituting expressions \eqref{bd_rates} for the birth and death rates, we get 

\begin{equation}\label{key} 
\begin{aligned}
L^*_n\ind{} &=  
\sum_{j=1}^d\sum_{x\in\TT^d_n} \frac{\etta_x}{\rho(1-\rho)}\lrch{ 
	(1 + \lambda \eta_{x - e_j}\eta_{x + e_j}) (1-\rho) - \rho}\\
&=\sum_{j=1}^d\sum_{x\in\TT^d_n} \frac{\etta_x}{\rho}\lrp{ 
	 \lambda\rho\,\etta_{x - e_j} + \lambda\rho\, \etta_{x + e_j} + \lambda\etta_{x-1}\etta_{x + e_j} }.
\end{aligned}
\end{equation}
Using the translation invariance of $ \TT^d_n $, it is straightforward to go from the expression above to the asserted formula \eqref{formula_ln*1}.

\end{proof}

\end{lemma}

	Recall that the birth and death rates defined by \eqref{birthdeathrate} and  \eqref{bd_rates} satisfy, by assumption, $ \int c_x\,\mathrm{d}\nu_{\rho} =0 $.
\begin{lemma}
 Let $ G:\{0,1\}^{\TT_n} \to \RR $ be a bounded function.  Then the following inequality holds:
	\begin{equation}\label{lemma_quadvar}
	\EE_{\nu_{\rho}}\lrc{\Big|
	\int_0^t \frac{1}{n}\sum_{x\in\TT_n} G_x \,c_x(\eta^n(s)) 
	\,\mathrm{d}s
	\Big|}\leq
	\frac{C\norm{G}_{\infty}}{\sqrt n},
	\end{equation}
	for some constant $ C = C(t) $. The same inequality holds with the mean-zero local function $ (\eta_{x+1}-\eta_x )^2 - \rho(1-\rho) $ in place of $ c_x(\eta) $.
\end{lemma}

\begin{proof}
	Denote by $ H_n(s) $ the relative entropy between the law of $ \eta^n_s $ and the product measure $ \nu_\rho $. Let $ \gamma >0 $. Applying the entropy inequality \eqref{entropy1} we can bound the lefthand side of \eqref{lemma_quadvar} by 
	\begin{equation}\label{quadvar_01}
	\frac{1}{\gamma}\int_0^t H_n(s)\,\mathrm{d}s + \frac{t}{\gamma} \log \int e^{ |\frac{\gamma}{n}\sum_{x\in\TT_n} G_x c_x |}\,\mathrm{d}\nu_{\rho}.
	\end{equation}

By Theorem \ref{entropy}, the first term in \eqref{quadvar_01} is bounded by $ Cte^{Ct} $ for some constant $ C $ that does not depend on $ n $.  The second term in \eqref{quadvar_01} is bounded from above by
\begin{equation}\label{key}
	\frac{t}{\gamma} \log \int e^{ \frac{\gamma}{n}\sum_{x\in\TT_n} G_x c_x } + e^{ -\frac{\gamma}{n}\sum_{x\in\TT_n} G_x c_x }\,\mathrm{d}\nu_{\rho}.
\end{equation}

By Hoeffding's inequality, each of the variables $ \tau_x \psi $ is $ \norm{\psi}_{\infty}-$subgaussian. In addition, the family $ \{ \tau_x \psi \}_{x\in \TT_n} $ has a finite-range dependency, say of range $ R$. It is possible to deduce that the previous expression is bounded by $ \frac{t}{\gamma} \lrp{ \log 2 + \frac{\gamma^2 (2R-1)\norm{G}^2_{\infty}}{n} }$. We omit the details because the same argument is used in the proofs of lemmas  \ref{conv_fluxo} and \ref{conv_caixa}. To finish the proof, one can choose $ \gamma = \sqrt n/\norm{G}_\infty $.

\end{proof}

\section{Concentration and entropy inequalities}\label{appendix_concentration}

Let $ \sigma >0 $. We say that a random variable $ X $ is $ \sigma^2 $-subgaussian if $ \EE\lrc{e^{\theta X}}\leq e^{\frac{\theta^2\sigma^2}{2}} $ for all $ \theta \in \RR $.

\begin{proposition}[Properties of subgaussian random variables]
	If $ X $ is a $ \sigma^2 $-subgaussian random variable, then the following inequalities hold:
	\begin{equation}\label{subgaussian_tail}
	\PP\lrp{|X| > \lambda} \leq 2e^{-\frac{\lambda^2}{2\sigma^2}}\mbox{ for all }\lambda >0
	\end{equation}
	and
\begin{equation}\label{chisquare_expmom}
\EE\lrc{ e^{cX^2} } \leq e^{8c\sigma^2}\mbox{ for all $ c \in (0, (4\sigma^2)^{-1} ] $.}
\end{equation}
\end{proposition}

\begin{proof}
	The first assertion uses Markov's inequality. For any $\theta > 0$,
	\begin{equation}
	\log P(X>\lambda)\leq \frac{\theta^2\sigma^2}{2}-\theta \lambda.
	\end{equation}
	
The expression on the righthand side attains its minimum at $\theta = \frac{\lambda}{\sigma^2}$, where it takes the value $-\lambda^2/2\sigma^2$. This computation shows that $ P(X>\lambda) \leq e^{-\frac{\lambda^2}{2\sigma^2} } $. In the same manner, one can show that $ P(-X>\lambda) \leq e^{-\frac{\lambda^2}{2\sigma^2} } $, thus obtaining \eqref{subgaussian_tail}.
	
For the second inequality, let $ c>0 $. Then
	
	\begin{equation}
	\begin{split}
	\EE[e^{c X^2}] & = 1 + \int_0^\infty 2c \, u \, e^{c u^2}\PP(|X|\geq u)\,du\\
	& \leq 1 + \int_0^\infty 4  c\, u \,  e^{-u^2 \lrp{\frac{1}{2\sigma^2} -c }  }\,du.
	\end{split}
	\end{equation}
When $ c \geq(2\sigma^2)^{-1} $, the integral above is infinite. When $ c < (2\sigma^2)^{-1} $, the integral can be computed explicitly. Assuming, as in the hypothesis, that $ c \leq (4\sigma^2)^{-1} $, we arrive at
\begin{equation}\label{key}
\begin{aligned}
\EE \lrc{ e^{c X^2} } & \leq 1 + \ \frac{2c}{\frac{1}{2\sigma^2} - c }\\
& \leq 1 + 8c\sigma^2 .
\end{aligned}
\end{equation}
An application of the inequality $ 1+ x \leq e^x $ then leads to \eqref{chisquare_expmom}.
	
\end{proof}

\begin{proposition}[Hoeffding's Inequality, \cite{lugosi}, Lemma 2.2]\label{hoeffding}
	Let $X$ be a mean zero random variable taking values in the interval $[a,b]$. Then $ X $ is $ \frac{(b-a)^2}{4} $-subgaussian.
%
%
\end{proposition}

%
%
%
%
%
%

\begin{proposition}[Bounded Differences Inequality, \cite{lugosi}, Theorem 6.2]\label{bounded_differences}
	Assume the function $f:\{0,1\}^{\TT_n}\to \RR$ satisfies
	\begin{equation}
	|f(\eta^x)-f(\eta)|\leq c_x
	\end{equation}
	for a family of constants $\{c_x: x\in \TT^d_n\}$. Then 
	\begin{equation}
	\log \nu_\rho \left(f(\eta)-\int f \,\mathrm{d}\nu_{\rho} > \delta \right) \leq -\frac{2 \delta^2}{\sum_{x\in \TT^d_n}c_x^2}.
	\end{equation}
\end{proposition}

\begin{lemma}[Partitioning the torus into $ k $-sparse sets]\label{torus_coloring}
	Let $ k\in \{1,\ldots, n\} $. There exists a partition $ \TT^d_n =\sqcup_{i\in \mathcal I} B_i $ such that each $ B_i $ is $ k $-sparse, meaning that if $ x,y \in B_i $ and $ x \neq y $ then $ \max_{j=1}^d |x_j-y_j| \geq k $. In addition, $ |\mathcal I| \leq (2k-1)^d $.
\end{lemma}

\begin{remark}
	The necessity to prove this technical lemma is the only complication brought about by our working with periodic boundary conditions. In all other parts of the proof, the assumption that the particles move on a torus helps to simplify the computations. 
\end{remark}

\begin{proof}
The proof is by induction on the dimension $ d $. When $ d=1 $ it is easy to write down the sets of the partition. Let $ n = mk + r $ where $ r\in \{0,1,\ldots, k-1 \} $ and $ m $ is a positive integer. Then $ \TT_n $ is the disjoint union of the $ k $ sets $ B_0, B_1, \ldots, B_{k-1} $, where $ B_j:= \{j, k+j, 2k+j, \ldots, (m-1)k + j \} $, and the $ r $ singletons $  \{mk + i \} \}  $ with $ 0\leq i \leq r-1$. By construction, each of the sets is $ k-$sparse. The number of sets in the partition is $ k + r \leq 2k-1$.

Now let $ d>1 $. Assume there exist partitions $ \TT^{d-1}_{n} = \sqcup_{i\in \mathcal I}B_i $ and $ \TT_n = \sqcup_{i'\in \mathcal I'}B_{i'} $, with $ |\mathcal I| \leq (2k-1)^{d-1} $ and $ |\mathcal I' | \leq 2k-1  $ of the torus into $ k-$sparse sets. Then the product partition $ \TT^{d}_n = \sqcup_{(i,i')\in \mathcal I \times \mathcal I'} B_i \times B_{i'} $ has at most $ (2k-1)^d $ sets and each of these sets is $ k-$sparse.
\end{proof}

\begin{lemma}\label{tail_lemma}
	Let $ X $ be a nonnegative random variable. Assume that 
	\begin{equation}
	\PP(|X|>\delta) \leq C/\delta^2
	\end{equation}
	for any $ \delta > 0 $. Then, for any $ \theta \in (0,2) $, there exists an universal constant $ C(\theta) $ such that $ \EE[|X|^\theta] \leq C(\theta)\cdot C^{\theta/2} $.
\end{lemma}

\begin{proof}
	Fix $\varepsilon >0$. Then
	
	\begin{equation*}
	\begin{split}
	\EE[X^{\theta}]
	& = \int_0^\infty \theta \delta^{\theta -
		1}\PP(X >\delta)\,\mathrm{d}\delta \\
	& \leq \varepsilon^\theta + \int_\varepsilon^\infty \theta
	C\delta^{\theta - 3}\,\mathrm{d}\delta \\
	& = \varepsilon^\theta +
	C\frac{\theta}{2-\theta}\varepsilon^{\theta - 2}.
	\end{split}
	\end{equation*}
	
	Choosing $\varepsilon = C^{1/2}$ we get $\EE[X^\theta]\leq (1+\theta/(2-\theta))C^{\theta/2}$.
\end{proof}

\begin{proposition}[\cite{kl}, Proposition A.8.2] Let $\mu$ and $\nu$ be probability measures on some finite set $\Omega$. Let $f:\Omega\to\RR$ be a function and $H(\mu|\nu)$ the relative entropy between $\mu$ and $\nu$. Then, for all $\gamma >0$,
	\begin{equation}\label{entropy1}
	\int f\,\mathrm{d}\mu \leq \frac{1}{\gamma}H(\mu|\nu) + \frac{1}{\gamma}\log \int e^{\gamma f}\,\mathrm{d}\nu
	\end{equation}
	and
	\begin{equation}\label{entropy2}
	\mu(A)\leq \frac{H(\mu|\nu)+\log 2}{\log\lrp{1+ \frac{1}{\nu(A)}}}.
	\end{equation}
\end{proposition}

\bibliographystyle{alpha}

\end{document}